\documentclass[11pt,a4paper]{article}

\usepackage[margin=1in]{geometry}
\usepackage{microtype}
\usepackage{setspace}
\usepackage[T1]{fontenc}
\usepackage[utf8]{inputenc}
\usepackage{lmodern}

\usepackage{amsmath,amssymb,amsthm,mathtools}
\usepackage{bm}            
\usepackage{mathrsfs}      
\usepackage{enumitem}
\usepackage{esint}      

\usepackage{amsopn}

\theoremstyle{plain}
\newtheorem{theorem}{Theorem}[section]

\newtheorem{lemma}[theorem]{Lemma}

\theoremstyle{definition}

\theoremstyle{remark}
\newtheorem{remark}[theorem]{Remark}

\usepackage{graphicx}
\usepackage{caption}
\usepackage{subcaption}
\usepackage{booktabs}    
\usepackage{tabularx}    
\usepackage{siunitx}     

\usepackage{tikz}
\usepackage{pgfplots}
\pgfplotsset{compat=1.18}

\usepackage[hidelinks]{hyperref}

\title{Quantitative Spectral Stability for an Embedded Annulus under Coupled Curve Shortening and $2$D Ricci Flows}
\author{MOHAMMADJAVAD HABIBIVOSTAKOLAEI \\ \small Institute of Mathematics, Henan Academy of Sciences \\ \texttt{mjhabibi@hnas.ac.cn}}
\date{} 

\begin{document}
\maketitle

\begin{abstract}
We study the spectral stability of Dirichlet eigenvalues on an embedded annulus whose boundary evolves by curve shortening flow while the ambient surface evolves under the two dimensional Ricci flow using variational formulas, Rellich--type identities, and harmonic capacity methods, we relate eigenvalue variations to geometric deficit and modulus. We establish quantitative bounds comparing the spectrum of the evolving annulus with that of a flat cylinder of equal modulus. As a consequence, we obtain geometric stability and a spectral gap estimate controlled by the deficit functional.
\end{abstract}

\tableofcontents
\bigskip

\section{Introduction and Motivation}
Consider $M$ as a smooth, compact, oriented Riemannian surface without boundary. It may carry either a fixed metric $g$, or a time--dependent metric $g\left(t\right)$ evolving by a $2$D Ricci flow: $\partial_{t}g = -2Kg$, where $K$ is the Gaussian curvature. $A$ denotes for an annulus embedded in $M$, bounded by two disjoint smooth simple closed curves $\Gamma_{0} \cup \Gamma_{1}$. In this manuscript, we are interested in a case that the boundary curves are evolving under curve shortening flow (CSF). Roughly speaking, at each time $t$, we pick a domain $A\left(t\right) \subset M$ diffeomorphic to an annulus, bounded by two disjoint smooth simple closed curves $\partial A\left(t\right) = \Gamma_{0}\left(t\right) \cup \Gamma_{1}\left(t\right)$, such that the boundary curves are evolving under CSF, i.e., 
\begin{align*}
\partial_{t}\Gamma_{i} = \kappa \nu, \,\,\,\,\, i=1,2,
\end{align*}
where $\kappa$ is geodesic curvature. This means, if $V$ denotes the outward normal velocity on $\partial A\left(t\right)$, then a point $x$ on $\partial A$, at each time $t$ it moves as $x_{t} = V\nu$, where $\nu$ is the outward unit normal of $A$.\\
During this manuscript, $u$ is harmonic capacity potential i.e., $u$ solves $\Delta_{g\left(t\right)}u = 0$ in $A\left(t\right)$, and it satisfies the boundary condition $u|_{\Gamma_{0}\left(t\right)} = 0$ and $u|_{\Gamma_{1}\left(t\right)} =1$. Also consider $E\left(t\right) = \frac{1}{2}\int_{A\left(t\right)}|\nabla u|^{2} d\mu_{g\left(t\right)}$ as an energy functional, where $h\left(t\right) = \frac{1}{2E\left(t\right)}$ is the modulus and $\mathcal{D}\left(t\right) = \frac{1}{2}\int_{A\left(t\right)}|\operatorname{Hess} u|^{2} d\mu_{g\left(t\right)}$ denotes the deficit functional. \\
For each $t$, let $\left(\lambda\left(t\right), \phi\left(t\right)\right)$ be a simple Dirichlet eigenpair for $-\Delta_{g\left(t\right)}\phi = \lambda \phi$ on $A\left(t\right)$, i.e.,
\begin{align*}
-\Delta_{g\left(t\right)}\phi = \lambda \phi \,\,\,\,\, \text{in}\,\,\, A\left(t\right),\,\,\,\,\,\, \phi|_{\partial A\left(t\right)} = 0,
\end{align*}
where we consider normalization 
\begin{align*}
\int_{A\left(t\right)}\phi^{2} d\mu_{g\left(t\right)} =1,
\end{align*}
to make sure of smoothly using of Kato--Rellich theory and analytic perturbations.

\subsection{Preliminaries}
The Ricci flow on a smooth Riemannian manifold $(M, g(t))$ is the evolution equation for the metric tensor given by
\begin{equation*}
\frac{\partial g_{ij}}{\partial t} = -2 \operatorname{Ric}_{ij},
\end{equation*}
where $\operatorname{Ric}$ denotes the Ricci curvature of $g(t)$. In two dimensions, this equation simplifies substantially because the Ricci tensor is proportional to the metric:
\begin{equation*}
\operatorname{Ric} = \frac{1}{2} R g,
\end{equation*}
where $R$ is the scalar curvature. Hence, the two-dimensional Ricci flow reduces to a scalar conformal evolution,
\begin{equation*}
\frac{\partial g}{\partial t} = - R g.
\end{equation*}
This structure implies that the flow preserves the conformal class of the metric, and the evolution of the conformal factor $u$ in $g(t) = e^{2u(t)} g_0$ satisfies the nonlinear heat-type equation
\begin{equation*}
\frac{\partial u}{\partial t} = \Delta_{g_0} u - K_{g_0},
\end{equation*}
where $K_{g_0}$ is the Gaussian curvature of the initial metric. In two dimensions, the Ricci flow acts as a curvature-normalizing process, tending to smooth out irregularities in the metric and evolve it toward one of constant curvature, depending on the topology of the underlying surface (positive, zero, or negative Euler characteristic). The flow is well-posed for all smooth initial metrics on compact surfaces and can be extended through singularities using appropriate rescaling techniques. This property makes the two-dimensional Ricci flow an essential analytical tool in geometric analysis, Teichm\"uller theory, and spectral geometry, especially when studying how geometric or spectral quantities evolve with time (see \cite{c1, h1}).\\
Let $\Gamma_t \subset (M, g(t))$ be a smooth one-parameter family of closed embedded curves on a Riemannian surface. The curve shortening flow (CSF) is defined by the geometric evolution equation
\begin{equation*}
\frac{\partial \gamma}{\partial t} = -\kappa \nu,
\end{equation*}
where $\gamma(s,t)$ is a smooth parametrization of $\Gamma_t$, $\kappa$ denotes the geodesic curvature of the evolving curve with respect to the ambient metric $g(t)$, and $\nu$ is the outward unit normal vector field along $\Gamma_t$. Intuitively, each point of the curve moves in the direction of its normal vector with speed equal to its geodesic curvature. The flow is the natural gradient flow of the length functional
\[
L[\Gamma] = \int_{\Gamma} ds,
\]
and hence monotonically decreases the total length of the curve:
\[
\frac{dL}{dt} = -\int_{\Gamma_t} \kappa^2 \, ds \le 0.
\]
Under the curve shortening flow, embedded curves remain embedded for all times up to the formation of singularities, and convex curves shrink smoothly to a round point in finite time. When the ambient metric evolves simultaneously under the Ricci flow, the combined system exhibits rich coupled geometric behavior, with the curvature of both the surface and the boundary curves interacting dynamically. On fixed or evolving Riemannian surfaces, the CSF provides a powerful tool for analyzing geometric and spectral quantities associated with evolving domains (for more details see \cite{g1,g2}).

\subsection{Background and Related Works}
The study of spectral properties of the Laplace--Beltrami operator on compact Riemannian surfaces and their subdomains forms a central theme in geometric analysis, linking curvature, topology, and eigenvalue behavior. Given a compact Riemannian surface $(M, g)$ (possibly with boundary), the eigenvalue problem
\begin{equation*}
-\Delta_g u = \lambda u,
\end{equation*}
with appropriate boundary conditions (Dirichlet, Neumann, or mixed) admits a discrete spectrum
\[
0 = \lambda_0 < \lambda_1 \le \lambda_2 \le \cdots \to \infty,
\]
where each eigenvalue $\lambda_k$ encodes geometric and topological information about $(M,g)$ (see \cite{c2}). Classical results show that these spectral quantities depend intricately on geometric invariants such as area, curvature, diameter, and boundary length. In the case of surfaces with boundary, geometric inequalities and comparison results (e.g., \cite{k1,y1}) relate $\lambda_1$ to the underlying conformal structure. For embedded annular regions, additional conformal parameters, such as the modulus of the annulus, play a fundamental role in determining the spectral behavior, since the Laplacian and its eigenfunctions are conformally covariant in two dimensions. The annulus thus serves as a natural model for studying how the interplay between geometry and topology influences spectral invariants.

When both the ambient metric and the boundary evolve, new analytical challenges arise. In particular, one may study how the eigenvalues of the Laplace--Beltrami operator on an embedded annulus $A_t \subset (M, g(t))$ evolve when the Riemannian metric $g(t)$ satisfies the two-dimensional Ricci flow and the boundary components $\partial A_t$ evolve under the curve shortening flow. Recently, Sobnack and Topping \cite{s1} established that for two disjoint closed embedded curves evolving by curve shortening flow on a surface, the conformal modulus of the annular region between them is monotonically increasing in time. This monotonicity of modulus under the flow provides a new conformally invariant quantity that can be used to control geometric and analytic behavior of evolving annuli. In particular, as the modulus increases, the annulus becomes "thinner'' in the conformal sense, which has direct implications for the variation of spectral quantities such as Dirichlet or Neumann eigenvalues. The combination of Ricci flow in the ambient metric and curve shortening flow on the boundary thus offers a rich framework for studying the coupled evolution of geometry and spectrum on dynamically changing domains.\\
In the past decade and especially in the last few years there has been a marked resurgence of interest in spectral problems for bounded regions of Riemannian manifolds, driven both by advances in geometric analysis and by renewed focus on boundary spectral problems such as the Steklov problem.  Survey and overview articles summarize many of the new directions: sharp upper and lower bounds for Steklov eigenvalues, asymptotics and counting function refinements, and connections to free boundary minimal surfaces have all seen significant progress.  Such advances place the study of eigenvalues on bounded domains in a broader program that links extremal eigenvalue problems to geometric PDE and variational constructions (see, e.g., recent surveys on Steklov problems and spectral inequalities).  
A second strand of recent work has focused on domains with holes (doubly connected regions, annuli) and how topology and conformal invariants influence spectral quantities.  For annular domains, classical results are being complemented by sharp bounds and rigidity phenomena for both Laplace and Steklov spectra: researchers have obtained explicit formulae and comparison results for symmetric annuli, studied monotonicity and extremal problems for eigenvalues as the inner hole moves, and explored the rigidity of eigenfunctions under boundary constraints.  These works show that conformal parameters of the annulus (notably the modulus) naturally enter eigenvalue estimates in two dimensions, because of the conformal covariance of the Laplacian and the special role of conformal mappings in controlling Rayleigh quotients on multiply connected domains.
Most pertinent to the present paper are very recent results that connect geometric flows of the boundary to conformal invariants of the enclosed region.  Sobnack and Topping \cite{s1} proved that when two nested embedded curves evolve simultaneously under the curve shortening flow the conformal modulus of the annulus between them is monotone (in fact increasing) in time; they further give an analogue on general ambient surfaces under a lower curvature bound.  This monotonicity provides a robust, flow-invariant control on the conformal geometry of evolving annuli and therefore supplies a natural bridge to spectral questions: an increasing modulus constrains the conformal class and hence has direct consequences for conformally-sensitive spectral quantities (e.g. Rayleigh quotients and the behavior of the Dirichlet/Neumann/Steklov eigenvalues) as the domain evolves.
Complementing this line, several recent preprints and articles have advanced eigenvalue bounds for domains with holes, provided explicit computations in symmetric annuli, and improved higher-order Steklov estimates that are sensitive to the presence of multiple boundary components.  Taken together these results suggest concrete strategies for studying the coupled Ricci–CSF evolution of an embedded annulus: use modulus monotonicity to control the conformal geometry of the annulus and then combine this control with evolving geometric quantities coming from the Ricci flow to derive monotonicity, comparison or asymptotic statements for the Laplace spectrum (under Dirichlet, Neumann, or Steklov conditions) on the evolving region.  We adopt this viewpoint in the sequel, relating evolving modulus bounds to spectral inequalities and deriving first results for the time-derivative of eigenvalues under the coupled flows (e.g, see more details in \cite{b1,d1,e1}).\\
Throughout the paper, we work under the following standing assumptions.
\begin{itemize}
\item $\left(M, g\left(t\right)\right)$ is a smooth compact Riemannian surface, with $g\left(t\right)$ evolving by the two-dimensional Ricci flow $\partial_{t}g = -2Kg$.
\item $A\left(t\right) \subset M$ is a smoothly embedded annulus whose boundary components $\Gamma_{0}\left(t\right), \Gamma_{1}\left(t\right)$ evolve by curve shortening flow.
\item The annulus admits a tubular neighborhood of fixed width $\rho_{0} >0$, uniformly in time, with uniformly bounded curvature and injectivity radius.
\item The Gaussian curvature satisfies $|K|_{L^{\infty}\left(A\left(t\right)\right)} \leq K_{0}$ for all $t$.
\item The harmonic capacity potential $u\left(t\right)$ satisfies the nondegeneracy condition
\begin{align*}
m:= \inf_{\partial A\left(t\right)} \partial_{\nu}u\left(t\right) >0.
\end{align*}
\end{itemize}
All constants appearing in the estimates depend only on $\left(\rho_{0}, K_{0}\right)$ and standard elliptic and Sobolev constants associated with $A\left(t\right)$, unless stated otherwise.

\section{Variational Formulas}
For a smooth family of regions $\Omega_t \subset M$ with boundary velocity field $V$, it asserts that for any sufficiently smooth function $F$,
\begin{equation*}
\frac{d}{dt} \int_{\Omega_t} F \, d\mu_g = \int_{\Omega_t} \frac{\partial F}{\partial t} \, d\mu_g + \int_{\partial \Omega_t} F\, (V \cdot \nu)\, d\sigma_g,
\end{equation*}
where $\nu$ is the outward unit normal and $d\sigma_g$ is the boundary measure induced by $g(t)$.  
This theorem underlies many variational identities in evolving geometry, allowing one to compute time derivatives of energy, area, or spectral functionals when both the metric and the domain evolve (e.g., see \cite{g3}). This formula is called "Reynolds transport theorem" and has been widely used in this section to find and compute variational formulas.\\
We recall that $\lambda$ is called an eigenvalue of Laplacian associated to eigenfunction $\phi$ with Dirichlet boundary condition when $-\Delta_{g\left(t\right)}\phi = \lambda \phi$ everywhere in the bounded region and $\phi$ vanished on the boundary. One may assume also the normalization i.e., for a bounded region $A$ then
\begin{align*}
\int_{A}\phi^{2} d\mu_{g} = 1.
\end{align*}
To express $\lambda$ as an integral, since $\int |\nabla \phi |^{2} = \int \phi\left(-\Delta \phi\right) = \lambda \int \phi^{2} = \lambda$, then
\begin{align*}
\lambda\left(t\right) = \int_{A\left(t\right)}|\nabla \phi |^{2} d\mu_{g}.
\end{align*}
\begin{lemma}\label{lemyek}
Let $\left(M, g\left(t\right)\right)$ be a smooth one--parameter family of Riemannian surfaces and let $A\left(t\right) \subset M$ be a smooth family of bounded domains diffeomorphic to an annulus with smooth boundary $\partial A\left(t\right) = \Gamma\left(t\right)$. If for each $t$, $\lambda\left(t\right)$ denotes Laplace's eigenvalue and $\phi\left(t\right)$ denotes its corresponding eigenfunction to be chosen satisfy Dirichlet boundary condition and normalization, then
\begin{align*}
\frac{d}{dt}\lambda = -\int_{\partial A}V\left(\partial_{\nu}\phi\right)^{2} d\sigma_{g} + \int_{A}\lbrace \frac{1}{2}\operatorname{tr}_{g}\left(\partial_{t}g\right)\left(|\nabla \phi |^{2} - \lambda \phi^{2}\right) - \partial_{t}g\langle \nabla \phi, \nabla \phi \rangle \rbrace d\mu_{g},
\end{align*}
where $V$ is the normal speed of the boundary in the outward normal direction $\nu$ and here $\partial_{\nu}\phi := \langle \nabla \phi, \nu \rangle_{g\left(t\right)}$.
\end{lemma}
\begin{proof}
Since 
\begin{align*}
\frac{d}{dt}\lambda = \frac{d}{dt}\int_{A\left(t\right)}|\nabla \phi |^{2} d\mu_{g},
\end{align*}
the problem becomes direct computation of the right--hand side. It is not hard to see that
\begin{align*}
\partial_{t}\left(|\nabla \phi |^{2}\right) = \partial_{t}\left(g^{ij}\partial_{i}\phi \partial_{j}\phi \right) = \left(\partial_{t} g^{ij}\right)\partial_{i}\phi \partial_{j}\phi + 2g^{ij}\partial_{i}\phi \partial_{j}\left(\partial_{t}\phi\right).
\end{align*}
Using $\partial_{t}g^{ij} = -g^{ik}g^{jl}\partial_{t}g_{kl}$, we may write
\begin{align*}
\left(\partial_{t}g^{ij}\right)\partial_{i}\phi \partial_{j}\phi = -\partial_{t}g\langle \nabla \phi, \nabla \phi \rangle,
\end{align*}
where we adopt the short--hand $\partial_{t}g\langle X, X\rangle = \partial_{t}g_{ij}X^{i}X^{j}$ for vector field $X$. So
\begin{align*}
\partial_{t}\left(|\nabla \phi |^{2}\right) = -\partial_{t}g\langle \nabla \phi, \nabla \phi \rangle + 2\langle \nabla \phi, \nabla\left(\partial_{t}\phi\right)\rangle.
\end{align*}
For easy--writing, since now, we suppress the explicit dependence on $t$ and write integrals over $A = A\left(t\right)$, $\partial A = \partial A\left(t\right)$. By applying Reynolds transport theorem to $|\nabla \phi |^{2}$, we get
\begin{align*}
\frac{d}{dt} \lambda &= \int_{A}\left(\partial_{t}\left(|\nabla \phi |^{2}\right) + \frac{1}{2}|\nabla \phi |^{2} \operatorname{tr}_{g}\left(\partial_{t}g\right)\right)d\mu_{g} + \int_{\partial A}|\nabla \phi |^{2}V d\sigma_{g} \\
&= \int_{A}\left(-\partial_{t}g\langle \nabla \phi, \nabla \phi \rangle +2\langle \nabla \phi, \nabla\left(\partial_{t}\phi\right)\rangle + \frac{1}{2}\operatorname{tr}_{g}\left(\partial_{t}g\right)|\nabla \phi |^{2}\right) d\mu_{g}\\
&+ \int_{\partial A}|\nabla \phi |^{2} V d\sigma_{g}.
\end{align*}
Use integration by parts (Green's identity) on the term $2\int_{A}\langle \nabla \phi, \nabla\left(\partial_{t}\phi\right)\rangle d\mu_{g}$, we see
\begin{align*}
2\int_{A}\langle \nabla \phi, \nabla \left(\partial_{t}\phi\right)\rangle d\mu_{g} = -2\int_{A}\left(\partial_{t}\phi\right)\Delta \phi d\mu_{g} + 2\int_{\partial A}\left(\partial_{t}\phi\right)\partial_{\nu}\phi d\sigma_{g},
\end{align*}
where $\partial_{\nu}\phi$ denotes the outward normal derivative on $\partial A$. Since $\phi$ satisfies the Laplace equation, thus
\begin{align*}
-2\int_{A}\left(\partial_{t}\phi\right) \Delta \phi d\mu_{g} = 2\lambda \int_{A}\phi \partial_{t}\phi d\mu_{g}.
\end{align*}
The eigenfunction $\phi$ satisfies the Dirichlet condition $\phi |_{\partial A\left(t\right)} \equiv 0$ for every $t$. Differentiate this identity following a boundary point $x\left(t\right) \in \partial A\left(t\right)$, we have
\begin{align*}
0 = \frac{d}{dt}\phi \left(x\left(t\right), t\right) = \partial_{t}\phi\left(x\left(t\right), t\right) + \langle \nabla \phi\left(x\left(t\right), t\right), x_{t}\rangle.
\end{align*}
But $x_{t} = V\nu$ and $\langle \nabla \phi, \nu\rangle = \partial_{\nu}\phi$, so
\begin{align*}
\partial_{t}|_{\partial A} = -V\partial_{\nu}\phi.
\end{align*}
(There is no tangential contribution because $\phi$ is fixed zero along the boundary). Thus the boundary term is
\begin{align*}
2\int_{\partial A}\left(\partial_{t}\phi\right) \partial_{\nu}\phi d\sigma_{g} = 2\int_{\partial A}\left(-V\partial_{\nu}\phi\right)\partial_{\nu}\phi d\sigma_{g} = -2\int_{\partial A}V\left(\partial_{\nu}\phi\right)^{2}d\sigma_{g}.
\end{align*}
We recall that the boundary term $\int_{\partial A}|\nabla \phi |^{2} V d\sigma_{g}$ is appeared from the Reynolds transport theorem, but $\phi |_{\partial A} = 0$ so tangential derivatives vanish, hence 
\begin{align*}
|\nabla \phi |^{2}|_{\partial A} = \left(\partial_{\nu}\phi\right)^{2}.
\end{align*}
Therefore, total boundary contribution is
\begin{align*}
-2\int_{\partial A}V\left(\partial_{\nu}\phi\right)^{2} d\sigma_{g} + \int_{\partial A}\left(\partial_{\nu}\phi\right)^{2}Vd\sigma_{g} = -\int_{\partial A}V\left(\partial_{\nu}\phi\right)^{2}d\sigma_{g}.
\end{align*}
By applying Reynolds transport theorem on normalization equation, we see
\begin{align*}
0 = \frac{d}{dt}\int_{A}\phi^{2} d\mu_{g} = \int_{A}\left(2\phi \partial_{t}\phi + \frac{1}{2}\phi^{2}\operatorname{tr}_{g}\left(\partial_{t}g\right)\right)d\mu_{g} + \int_{\partial A}\phi^{2}V d\sigma_{g}.
\end{align*}
Similarly, because of the boundary condition, the boundary term vanishes. Thus
\begin{align}\label{yek}
\int_{A}\phi \partial_{t}\phi d\mu_{g} = -\frac{1}{4}\int_{A}\phi^{2}\operatorname{tr}_{g}\left(\partial_{t}g\right)d\mu_{g}.
\end{align}
This will help to eliminate $\int \phi \partial_{t}\phi$. Now, Gathering the pieces from above, we see
\begin{align*}
\frac{d}{dt}\lambda &= - \int_{\partial A}V\left(\partial_{\nu}\phi\right)^{2} d\sigma_{g} + \int_{A}\left(-\partial_{t}g\langle \nabla \phi, \nabla \phi\rangle + 2\langle \nabla \phi, \nabla\left(\partial_{t}\phi\right)\rangle + \frac{1}{2}\operatorname{tr}_{g}\left(\partial_{t}g\right)|\nabla \phi |^{2}\right) d\mu_{g}\\
&= -\int_{\partial A}V\left(\partial_{\nu}\phi\right)^{2} d\sigma_{g} + \int_{A}\left(-\partial_{t}g\langle \nabla \phi, \nabla \phi \rangle + 2\lambda \phi \partial_{t}\phi + \frac{1}{2}\operatorname{tr}_{g}\left(\partial_{t}g\right)|\nabla \phi |^{2}\right) d\mu_{g}.
\end{align*}
From (\ref{yek}), we have
\begin{align*}
2\lambda\int_{A}\phi \partial_{t}\phi d\mu_{g} = 2\lambda\left(-\frac{1}{4}\int_{A}\phi^{2}\operatorname{tr}_{g}\left(\partial_{t}g\right)d\mu_{g}\right)
=\frac{\lambda}{2}\int_{A}\phi^{2}\operatorname{tr}_{g}\left(\partial_{t}g\right)d\mu_{g}.
\end{align*}
Thus, the interior integrand becomes
\begin{align*}
\frac{1}{2}\operatorname{tr}_{g}\left(\partial_{t}g\right)\left(|\nabla \phi |^{2} - \lambda \phi^{2}\right) - \partial_{t}g\langle \nabla \phi, \nabla \phi \rangle.
\end{align*}
This with the boundary integrand is exactly what we were looking for.
\end{proof}
\begin{remark}[{\bf Special Cases}]
\begin{itemize}
\item Consider fixed metric $\left(\partial_{t}g \equiv 0\right)$. Then the formula reduces to the classical Hadamard formula
\begin{align*}
\frac{d}{dt}\lambda = -\int_{\partial A}V\left(\partial_{\nu}\phi\right)^{2} d\sigma_{g},
\end{align*}
which is standard, i.e., inward motion $\left(V < 0\right)$ increases $\lambda$.
\item Consider fixed domain $\left(V \equiv 0\right)$, thus metric varying only. Then
\begin{align*}
\frac{d}{dt}\lambda = \int_{A}\left(-\partial_{t}g\langle \nabla \phi, \nabla \phi\rangle + \frac{1}{2}\operatorname{tr}_{g}\left(\partial_{t}g\right)|\nabla \phi |^{2} - \frac{\lambda}{2}\operatorname{tr}_{g}\left(\partial_{t}g\right)\phi^{2}\right)d\mu_{g}.
\end{align*}
\item Conformal variation $\partial_{t}g = 2\dot{f}g$ (infinitesimal conformal factor $2\dot{f}$) in dimension $n$. Computing $\partial_{t}g\langle \nabla \phi, \nabla \phi \rangle = 2\dot{f}|\nabla \phi |^{2}$, and $\operatorname{tr}_{g}\left(\partial_{t}g\right) = 2n\dot{f}$. Plugging into the interior integrand yeilds
\begin{align*}
-2\dot{f}|\nabla \phi |^{2} + n\dot{f}|\nabla \phi |^{2} - n\dot{f}\lambda \phi^{2} = \left(n-2\right)\dot{f}|\nabla \phi |^{2} - n\dot{f}\lambda \phi^{2}.
\end{align*}
In dimension $n=2$, this simplifies to $-2\dot{f}\lambda \phi^{2}$ integrated, hence
\begin{align*}
\frac{d}{dt}\lambda = -2\lambda\int_{A}\dot{f}\phi^{2} d\mu_{g},
\end{align*}
and if $\dot{f}$ is constant this gives the expected scaling law $\lambda$ scales by factor $e^{-2f}$.
\item Consider Ricci flow in $2$D, $\partial_{t}g = -2Kg$ where $K$ is the Gaussian curvature. Then $\partial_{t}g\langle \nabla \phi, \nabla \phi \rangle = -2K|\nabla \phi |^{2}$ and since $n=2$, thus $\operatorname{tr}_{g}\left(\partial_{t}g\right) = -4K$. Then we have
\begin{align*}
&-\partial_{t}g\langle \nabla \phi, \nabla \phi\rangle + \frac{1}{2}\operatorname{tr}_{g}\left(\partial_{t}g\right)|\nabla \phi |^{2} - \frac{\lambda}{2}\operatorname{tr}_{g}\left(\partial_{t}g\right)\phi^{2}\\
&= -\left(-2K|\nabla \phi |^{2}\right) + \frac{1}{2}\left(-4K\right)|\nabla \phi |^{2} - \frac{\lambda}{2}\left(-4K\right)\phi^{2}\\
&= 2K|\nabla \phi |^{2} - 2K|\nabla \phi |^{2} + 2\lambda K \phi^{2} = 2\lambda K \phi^{2}.
\end{align*}
Hence, under curve shortening flow (CSF) and Ricci flow the formula becomes
\begin{align*}
\frac{d}{dt}\lambda = -\int_{\partial A}V\left(\partial_{\nu}\phi\right)^{2} d\sigma_{g} + 2\lambda\int_{A}K\phi^{2}d\mu_{g}.
\end{align*}
\end{itemize}
\end{remark}
\begin{remark}
The differentiability of $\lambda(t)$ and $\phi(t)$ follows from classical analytic perturbation theory for simple eigenvalues (Kato--Rellich theory), together with smooth dependence of the domain and metric. The normalization $\int_{A(t)}\phi^2\,d\mu_{g(t)}=1$ ensures uniqueness of the eigenfunction up to sign.
\end{remark}
The infinitesimal conformal factor means: if the metric varies in time by a conformal scaling $g\left(t\right) = e^{2f\left(t\right)}g_{0}$, then $\partial_{t}g = 2\dot{f}g$. The scalar function $\dot{f}\left(x\right)$ is the "infinitesimal conformal factor". It measures how fast the metric is being stretched uniformly in all directions at each point. 
\begin{lemma}\label{lemdo}
Consider $\left(M,g\right)$ as a smooth $2$-dimensional Riemannian manifold and let for each $t$, $A\left(t\right) \subset M$ be a smooth domain diffeomorphic to an annulus bounded by $\partial A\left(t\right) = \Gamma_{0}\left(t\right) \cup \Gamma_{1}\left(t\right)$. Also let $\nu$ denote the outward unit normal of $A\left(t\right)$ along the $\partial A\left(t\right)$; the boundary moves with normal velocity $V$ relative to $\nu$, i.e., a boundary point follows $x_{t} = V\nu$ (so if the boundary moves inward then $V<0$). If $u$ and $E$ denote the harmonic capacity potential on $A\left(t\right)$ and capacity energy respectively, then under the same Dirichlet assumptions as Lemma \ref{lemyek}, we have
\begin{align*}
\frac{d}{dt}E = -\frac{1}{2}\int_{A}\left(|\operatorname{Hess}u|^{2} + K|\nabla u|^{2} + \partial_{t}g\left( \nabla u, \nabla u \right) - \frac{1}{2}|\nabla u|^{2} \operatorname{tr}_{g}\left(\partial_{t}g\right)\right) d\mu_{g},
\end{align*}
where $K$ is the Gaussian curvature and $\operatorname{tr}_{g}\left(\partial_{t}g\right) = g^{ij}\partial_{t}g_{ij}$.
\end{lemma}
\begin{proof}
We recall that
\begin{align*}
E\left(t\right) = \frac{1}{2}\int_{A}|\nabla u|^{2} d\mu_{g}.
\end{align*}
Apply Reynolds transport theorem to $F = \frac{1}{2}|\nabla u|^{2}$. To compute $\partial_{t}\left(\frac{1}{2}|\nabla u|^{2}\right)$, we differentiate the pointwise quantity
\begin{align*}
\frac{1}{2}|\nabla u|^{2} = \frac{1}{2}g^{ij}\partial_{i}u\partial_{j}u.
\end{align*}
Differentiate (holding coordinate fixed)
\begin{align*}
\partial_{t}\left(\frac{1}{2}|\nabla u|^{2}\right) = \frac{1}{2}\left(\partial_{t}g^{ij}\right)\partial_{i}u\partial_{j}u + g^{ij}\partial_{i}u\partial_{j}\left(\partial_{t}u\right).
\end{align*}
Use $\partial_{t}g^{ij} = -g^{ik}g^{jl}\partial_{t}g_{kl}$ to write the first term as $-\frac{1}{2}\partial_{t}g\left( \nabla u, \nabla u\right)$. Thus
\begin{align*}
\partial_{t}\left(\frac{1}{2}|\nabla u|^{2}\right) = -\frac{1}{2}\partial_{t}g\left( \nabla u, \nabla u\right) + \langle \nabla u, \nabla\left(\partial_{t}u\right)\rangle.
\end{align*}
By plugging into Reynolds theorem, we get
\begin{align}\label{do}
\frac{dE}{dt} = -\frac{1}{2}\int_{A}\left(\partial_{t}g\left( \nabla u, \nabla u\right) + \langle \nabla u, \nabla\left(\partial_{t}u\right)\rangle - \frac{1}{2}|\nabla u|^{2}\operatorname{tr}_{g}\left(\partial_{t}g\right)\right)d\mu_{g} + \frac{1}{2}\int_{\partial A}|\nabla u|^{2}Vd\sigma_{g}.
\end{align}
Apply Green's identity (integrate by parts) to $\langle \nabla u, \nabla\left(\partial_{t}u\right)\rangle$, we see
\begin{align*}
\int_{A}\langle \nabla u, \nabla\left(\partial_{t}u\right)\rangle d\mu_{g} = -\int_{A}\left(\partial_{t}u\right)\Delta_{g}u d\mu_{g} + \int_{\partial A}\left(\partial_{t}u\right)\partial_{\nu}u d\sigma_{g}.
\end{align*}
From harmonicity ($\Delta_{g}u = 0, \,\, \text{in} A\left(t\right)$), the volume term vanishes and we are left with only the boundary term
\begin{align}\label{se}
\int_{A}\langle \nabla u, \nabla\left(\partial_{t}u\right)\rangle d\mu_{g} = \int_{\partial A}\left(\partial_{t}u\right)\partial_{\nu}u d\sigma_{g}.
\end{align}
Since $u$ has constant Dirichlet values on each components of the boundary, differentiating the identity $u\left(x\left(t\right), t\right) = \text{const}$ along a boundary point $x\left(t\right)$ gives
\begin{align*}
0 = \frac{d}{dt}u\left(x\left(t\right), t\right) = \partial_{t}u + \langle \nabla u, x_{t}\rangle = \partial_{t}u + V\partial_{\nu}u.
\end{align*}
Thus on $\partial A$,
\begin{align}\label{char}
\partial_{t}u|_{\partial A} = -V\partial_{\nu}u.
\end{align}
Substitute (\ref{char}) in (\ref{se}), we get
\begin{align}\label{pang}
\int_{A}\langle \nabla u, \nabla\left(\partial_{t}u\right)\rangle d\mu_{g} = \int_{\partial A}\left(-V\partial_{\nu}u\right)\partial_{\nu}u d\sigma_{g} = - \int_{\partial A}V\left(\partial_{\nu}u\right)^{2}d\sigma_{g}.
\end{align}
Return to (\ref{do}) and (\ref{pang}), the two boundary contributions are $\frac{1}{2}\int_{\partial A}|\nabla u|^{2}Vd\sigma_{g}$ and $-\int_{\partial A}V\left(\partial_{\nu}u\right)^{2}d\sigma_{g}$. But on $\partial A$ the tangential derivative of $u$ vanishes (since $u$ is constant there), hence $|\nabla u|^{2}|_{\partial A} = \left(\partial_{\nu}u\right)^{2}$. Therefore, the sum of boundary contributions equals
\begin{align}\label{shesh}
\frac{1}{2}\int_{\partial A}\left(\partial_{\nu}u\right)^{2}V d\sigma_{g} - \int_{\partial A}\left(\partial_{\nu}u\right)^{2}V d\sigma_{g} = -\frac{1}{2}\int_{\partial A}V\left(\partial_{\nu}u\right)^{2}d\sigma_{g}.
\end{align}
So far we have converted (\ref{do}) into
\begin{align}\label{haft}
\frac{dE}{dt} = -\frac{1}{2}\int_{A}\left(\partial_{t}g\left( \nabla u, \nabla u\right) -\frac{1}{2}|\nabla u|^{2}\operatorname{tr}_{g}\left(\partial_{t}g\right)\right) d\mu_{g} - \frac{1}{2}\int_{\partial A}V\left(\partial_{\nu}u\right)^{2} d\sigma_{g}.
\end{align}
This is already a corrected exact formula; the derivative of energy equals a purely metric interior term plus a boundary shape term. Next, we should convert the boundary term into an interior integral.\\
We recall that on a Riemannian manifold, the Bochner identity for a smooth function $u$ is
\begin{align*}
\Delta\left(\frac{1}{2}|\nabla u|^{2}\right) = |\operatorname{Hess}u|^{2} + \langle \nabla u, \nabla\left(\Delta u\right)\rangle + \operatorname{Ric}\left( \nabla u, \nabla u\right).
\end{align*}
In dimension $2$, the Ricci curvature satisfies $\operatorname{Ric} = Kg$, so for our harmonic potential $\left(\Delta u = 0\right)$ this simplifies to
\begin{align}\label{hasht}
\Delta\left(\frac{1}{2}|\nabla u|^{2}\right) = |\operatorname{Hess}u|^{2} + K|\nabla u|^{2}.
\end{align}
Integrate (\ref{hasht}) over $A$ and apply the divergence theorem;
\begin{align}\label{noh}
\int_{A}\left(|\operatorname{Hess}u|^{2} + K|\nabla u|^{2}\right) d\mu_{g} = \int_{A}\Delta\left(\frac{1}{2}|\nabla u|^{2}\right) d\mu_{g} = \int_{\partial A}\frac{1}{2}\partial_{\nu}\left(|\nabla u|^{2}\right) d\sigma_{g}.
\end{align}
Thus the left--hand side interior quantity can be expressed as a single boundary normal derivative integral. On $\partial A$ the tangential derivative of $u$ vanishes (as $u$ is constant on each component boundary), so $\nabla u|_{\partial A} = \left(\partial_{\nu}u\right)\nu$. Therefore,
\begin{align*}
\partial_{\nu}\left(|\nabla u|^{2}\right) = \partial_{\nu}\left(\left(\partial_{\nu}u\right)^{2}\right) = 2\left(\partial_{\nu}u\right)\partial_{\nu}\left(\partial_{\nu}u\right) = 2\left(\partial_{\nu}u\right)\operatorname{Hess}u\left(\nu, \nu\right).
\end{align*}
Hence (\ref{noh}) yields
\begin{align}\label{dah}
\int_{A}\left(|\operatorname{Hess}u|^{2} + K|\nabla u|^{2}\right)d\mu_{g} = \int_{\partial A}\left(\partial_{\nu}u\right)\operatorname{Hess}u\left(\nu, \nu\right) d\sigma_{g}.
\end{align}
At a boundary point, one may write local coordinates $\left(s, \mathfrak{n}\right)$ where $s$ is arc length along $\partial A$ and $\mathfrak{n}$ is signed distance in the normal direction (so $\nu = \partial_{\mathfrak{n}}$). A standard decomposition of the Laplacian near the boundary is
\begin{align*}
\Delta u = \partial^{2}_{\mathfrak{n}}u + \kappa\partial_{\mathfrak{n}}u + \partial^{2}_{s}u,
\end{align*}
where $\kappa$ is the geodesic curvature of $\partial A$ in $M$ (sign convention: $\kappa >0$ for boundary curving toward the interior). On the boundary $u|_{\partial A}$ is constant on each component, hence $\partial_{s}u =0$ and therefore $\partial_{s}^{2}u =0$ there. Using $\Delta u =0$, we obtain at boundary points 
\begin{align*}
\partial^{2}_{\mathfrak{n}}u + \kappa\partial_{\mathfrak{n}}u = 0,
\end{align*}
i.e.,
\begin{align}\label{yazdah}
\operatorname{Hess}u\left(\nu, \nu\right)|_{\partial A} = -\kappa\partial_{\nu}u.
\end{align}
Substitute (\ref{yazdah}) in (\ref{dah}), we get
\begin{align*}
\int_{A}\left(\operatorname{Hess}u|^{2} + K|\nabla u|^{2}\right)d\mu_{g} = \int_{\partial A}\left(\partial_{\nu}u\right)\left(-\kappa\partial_{\nu}u\right)d\sigma_{g} = -\int_{\partial A}\kappa\left(\partial_{\nu}u\right)^{2}d\sigma_{g},
\end{align*}
or equivalently
\begin{align}\label{davazdah}
\int_{\partial A}\kappa\left(\partial_{\nu}u\right)^{2}d\sigma_{g} = -\int_{A}\left(|\operatorname{Hess}u|^{2} + K|\nabla u|^{2}\right) d\mu_{g}.
\end{align}
Curve shortening flow (the usual geometric CSF) moves each boundary curve with normal velocity equal to the curvature in the inward normal direction
\begin{align*}
\partial_{t}x = \kappa\nu.
\end{align*}
Our sign convention uses $\nu = \text{outward normal}$, so the outward normal velocity $V$ of CSF equals $V = -\kappa$ (CSF moves the boundary inward with speed $\kappa$, thus $V<0$ when $\kappa >0$). Thus for CSF we set
\begin{align*}
V = -\kappa.
\end{align*}
Insert this into the boundary term of (\ref{haft}), we see
\begin{align*}
-\frac{1}{2}\int_{\partial A}V\left(\partial_{\nu}u\right)^{2} d\sigma_{g} = -\frac{1}{2}\int_{\partial A}\left(-\kappa\right)\left(\partial_{\nu}u\right)^{2}d\sigma_{g} = \frac{1}{2}\int_{\partial A}\kappa\left(\partial_{\nu}u\right)^{2}d\sigma_{g}.
\end{align*}
Use (\ref{davazdah}) to replace the boundary integral by the interior integral
\begin{align}\label{sizdah}
\frac{1}{2}\int_{\partial A}\kappa\left(\partial_{\nu}u\right)^{2}d\sigma_{g} = -\frac{1}{2}\int_{A}\left(|\operatorname{Hess}u|^{2} + K|\nabla u|^{2}\right)d\mu_{g}.
\end{align}
Therefore, combining (\ref{haft}) with (\ref{sizdah}), we obtain the general formula for the derivative of energy when the boundary moves by CSF
\begin{align*}
\frac{dE}{dt} = -\frac{1}{2}\int_{A}\left(|\operatorname{Hess}u|^{2} + K|\nabla u|^{2} + \partial_{t}g\left(\nabla u, \nabla u\right) - \frac{1}{2}|\nabla u|^{2}\operatorname{tr}_{g}\left(\partial_{t}g\right)\right)d\mu_{g}.
\end{align*}
\end{proof}
In the above proof, $\eta$ signed distance in the normal direction for a point $x$ near the boundary, $\eta\left(x\right)$ is the geodesic distance from $x$ to the boundary taken positive outward. Thus $\left(\mathrm{s}, \eta\right)$ from local "Fermi coordinates" near $\partial A$: $\mathrm{s}$ runs tangentially along the boundary and $\eta$ measures how far you move in the normal direction.
\begin{remark}[{\bf Some specialisations }]
\begin{itemize}
\item One may prefer to display the interior metric term as the contraction of $\partial_{t}g$ with a trace free symmetric tensor. Thus the final integrand can be rewritten as
\begin{align*}
-\frac{1}{2}\left(\partial_{t}g\left(\nabla u, \nabla u\right) - \frac{1}{2}|\nabla u|^{2}\operatorname{tr}_{g}\left(\partial_{t}g\right)\right) = -\frac{1}{4}\langle \partial_{t}g, 2\nabla u \otimes \nabla u -|\nabla u|^{2}g\rangle.
\end{align*}
\item If $\partial_{t}g =0$ (Fixed metric) and the boundary moves by CSF (so $V = -\kappa$), the final integral reduces to the Topping's identity (in Riemannian case) \cite{s1}
\begin{align*}
\frac{dE}{dt} = -\frac{1}{2}\int_{A}\left(|\operatorname{Hess}u|^{2} + K|\nabla u|^{2}\right)d\mu_{g}.
\end{align*}
This shows $\dot{E} \leq 0$, if the surface has non--negative Gaussian curvature. In particular on a flat surface $\left( K \equiv 0\right)$, $E$ decreases and the deficit $\frac{1}{2}\int |\operatorname{Hess}u|^{2} $ controls the rate. 
\item Take $\partial_{t}g = -2Kg$ (metric evolves by $2$D Ricci flow and boundary curves evolve under CSF). Thus, the metric--variation integrands become
\begin{itemize}
\item[a)] $\partial_{t}g\left(\nabla u, \nabla u\right) = -2K|\nabla u|^{2}$.
\item[b)] $\operatorname{tr}_{g}\left(\partial_{t}g\right) = -4K$, (since $\operatorname{tr}_{g}\left(g\right) =2$ in $2$D) 
\end{itemize}
Therefore, the final identity becomes
\begin{align*}
\frac{dE}{dt} = -\frac{1}{2}\int_{A}\left(|\operatorname{Hess}u|^{2} + K|\nabla u|^{2}\right) d\mu_{g}.
\end{align*}
It means that, the Ricci flow term cancels the explicit metric--variation term in dimension two, so the previous Topping's identity remains valid even when the ambient metric evolves by the Ricci flow, provided the boundary moves by curve shortening flow.
\end{itemize}
\end{remark}

\section{Spectral Comparison of an Annulus and a Cylinder}
\subsection{Rellich--type Formula and the Deficit Boundary Term}
Define the vector field (on $A$)
\begin{align*}
Y:= |\nabla \phi |^{2}X - 2\left(X. \nabla \phi\right)\nabla \phi.
\end{align*}
Here $X$ is a vector field such that $X. \nabla \phi := \langle X, \nabla \phi\rangle$ (we may consider $X$ is every extension of $\mathcal{V}\nu$). The reason to consider $Y$ is that its boundary flux produces the boundary quantity we want, while its divergence can be expressed in interior terms that involve $\nabla X$ and $\nabla \phi$. By computing $\operatorname{div}Y$ and integrate it over $A$, one has the standard Rellich--type identity on a Riemannian manifold \cite{b2,r1}
\begin{align}\label{chardah}
\int_{\partial A}\langle X, \nu\rangle \left(\partial_{\nu}\phi\right)^{2} d\sigma_{g} = -\int_{A}\lbrace\left(\operatorname{div}X\right)|\nabla \phi |^{2} - 2\langle \nabla X, \nabla \phi \otimes \nabla \phi \rangle - 2\left(X. \nabla \phi\right)\Delta \phi\rbrace d\mu_{g}.
\end{align} 
Since $-\Delta \phi = \lambda \phi$, then
\begin{align*}
-2\left(X. \nabla \phi\right)\Delta \phi = 2\lambda \phi \left(X.\nabla \phi\right),
\end{align*}
thus we observe that
\begin{align*}
2\lambda \phi\left(X. \nabla \phi\right) = \lambda X.\nabla \left(\phi^{2}\right) = \lambda \operatorname{div}\left(\phi^{2}X\right) - \lambda \phi^{2} \operatorname{div}X.
\end{align*}
Integrate over $A$ (the boundary term vanishes), we have
\begin{align*}
\int_{A}2\lambda \phi\left(X.\nabla \phi\right)d\mu_{g} = -\lambda \int_{A}\phi^{2}\operatorname{div}X d\mu_{g}.
\end{align*}
Substituting into the (\ref{chardah}) gives the compact form
\begin{align}\label{panzdah}
\int_{\partial A}\langle X, \nu\rangle\left(\partial_{\nu}\phi\right)^{2} d\sigma_{g} = -\int_{A}\lbrace\left(\operatorname{div}X\right)\left(|\nabla \phi |^{2} - \lambda \phi^{2}\right) -2\langle \nabla X, \nabla \phi \otimes \nabla \phi \rangle \rbrace d\mu_{g}.
\end{align}
Choose any smooth extension $X$ of the boundary vector field $\mathcal{V}\nu$ into the interior, and bound the right--hand side in terms of $L^{\infty}$ norms of $\operatorname{div}X$ and $\nabla X$ and interior $L^{1}$--type quantities in $\phi$ (we will later replace this extension by the smooth extension inside the collar). From (\ref{panzdah}) we have
\begin{align}\label{shanzdah}
|\int_{\partial A}\mathcal{V}\left(\partial_{\nu}\phi\right)^{2}d\sigma_{g}| &= |\int_{A}\left(\operatorname{div}X\right)\left(|\nabla \phi |^{2} - \lambda \phi^{2}\right)d\mu_{g} - 2\int_{A}\langle \nabla X, \nabla \phi \otimes \nabla \phi \rangle d\mu_{g}| \nonumber\\
&\leq \|\operatorname{div} X \|_{L^{\infty}\left(A\right)}\int_{A}| |\nabla \phi |^{2} - \lambda \phi^{2}|d\mu_{g}\nonumber\\
&+ 2\|\nabla X\|_{L^{\infty}\left(A\right)}\int_{A}|\nabla \phi |^{2} d\mu_{g}.
\end{align} 
By using eigen--relation (normalization of $\phi$), one can see
\begin{align*}
\int_{A}| |\nabla \phi |^{2} - \lambda \phi^{2}|d\mu_{g} \leq \int_{A}|\nabla \phi |^{2}d\mu_{g} + \lambda \int_{A}\phi^{2} d\mu_{g} = 2\lambda.
\end{align*}
Therefore, (\ref{shanzdah}) yields the simple bound
\begin{align*}
|\int_{\partial A}\mathcal{V}\left(\partial_{\nu}\phi\right)^{2}| \leq 2\lambda\left(\|\operatorname{div}X\|_{L^{\infty}\left(A\right)} + \|\nabla X\|_{L^{\infty}\left(A\right)}\right).
\end{align*}
Here in the extension of $X$, $\mathcal{V}: \partial A \rightarrow \mathbb{R}$ is a general (arbitrary) boundary weight and is different from boundary normal velocity in geometric flows.
\begin{remark}[{\bf Few Points on Extension $X$}]
We need $X$ such that $X|_{\partial A} = \mathcal{V}\nu$ (so $\langle X, \nu\rangle = \mathcal{V}$ on $\partial A$). A convenient constructive choice uses a tubular (Fermi) neighborhood of $\partial A$.
\begin{itemize}
\item Let $\rho > 0$ be smaller than the injectivity radius of $\partial A$ in $M$. In the tubular neighborhood $U_{\rho} = \lbrace p \in A| \operatorname{dist}_{g}\left(p, \partial A\right) < \rho \rbrace$ introduce normal coordinates $\left(s, \mathfrak{n}\right)$ where $s$ is arc--length along $\partial A$, $\operatorname{dist}_{g}$ means the Riemannian distance induced by Riemannian metric $g$, and $\mathfrak{n}$ is signed distance along the normal, i.e., $\mathfrak{n} =0$ on $\partial A$, $\mathfrak{n}>0$ pointing toward the interior (or outward--pick consistent sign). In these coordinates $\nu = \partial_{\mathfrak{n}}$, and the metric takes the well-known form $g = d\mathfrak{n}^{2} + a\left(\mathfrak{n}, s\right)^{2}ds^{2}$ with $a\left(0,s\right) =1$ and $\partial_{\mathfrak{n}}a\left(0,s\right) = -\kappa\left(s\right)$ etc. All geometric derivatives of $a$ are controlled by curvature.
\item Choose a smooth cut--off function $\zeta\left(\mathfrak{n}\right)$ with $\zeta\left(0\right) =1$, $\zeta\left(\mathfrak{n}\right) =0$, for $\mathfrak{n} \geq \rho$, and $|\zeta^{\prime}| \lesssim \rho^{-1}$. Define the extension in the collar 
\begin{align*}
X\left(s,\mathfrak{n}\right) = \zeta\left(\mathfrak{n}\right)\mathcal{V}\left(s\right)\nu\left(s,\mathfrak{n}\right),
\end{align*}
where $\mathcal{V}\left(s\right)$ is the given boundary scalar (depends on boundary point). Extend $X$ by zero outside the collar.
\item Careful estimates (standard in geometry) give, for some constant $C$ depending only on $\rho$ and the ambient curvature bounds and the $C^{1}$--geometry of $\partial A$,
\begin{align*}
\|\nabla X\|_{L^{\infty}\left(A\right)} \leq C\left(\rho^{-1}\|\mathcal{V}\|_{L^{\infty}\left(\partial A\right)} + \|\partial_{s}\mathcal{V}\|_{L^{\infty}\left(\partial A\right)} + \|\kappa \mathcal{V}\|_{L^{\infty}\left(\partial A\right)}\right).
\end{align*}
Thus the sub--norm of $\nabla X$ (similarly $\operatorname{div}X$) reduces to boundary norms of $\mathcal{V}$ and its tangential derivative $\partial_{s}V$--exactly the geometric quantities available if $\mathcal{V}$ equals $\left(-\kappa\right)$ or similarly, plugging these bound into Rellich--type identity, yields a concrete inequality of the form
\begin{align*}
|\int_{\partial A}\mathcal{V}\left(\partial_{\nu}\phi\right)^{2} d\sigma | \leq C\lambda\left(\rho^{-1}\|\mathcal{V}\|_{L^{\infty}\left(\partial A\right)} + \|\partial_{s}\mathcal{V}\|_{L^{\infty}\left(\partial A\right)} + \|\kappa \mathcal{V}\|_{L^{\infty}\left(\partial A\right)}\right),
\end{align*}
where $C$ depends only on the geometry of the collar (curvature bounds, injectivity radius).
\end{itemize}
\end{remark}
We recall the extension of the vector field $X$, where we want $\langle X, \nu\rangle = \mathcal{V}$ on $\partial A$. A convenient (and standard) choice that allows insertion of $\operatorname{Hess}u$ is to take $X$ proportional to $\nabla u$ in a collar neighborhood and vanishing away from the collar. Precisely, let $\mathfrak{n}$, $U_{\rho}$, $s$, and $\zeta$ be as same as above, then define the boundary function
\begin{align*}
\alpha\left(s\right) := \frac{\mathcal{V}\left(s\right)}{\partial_{\nu}u\left(s\right)}, \,\,\,\,\, s \in \partial A.
\end{align*}
If we consider the nondegeneracy assumption, i.e., $m:= \inf_{\partial A} \partial_{\nu}u >0$, then $\alpha$ is well--defined and bounded,
\begin{align*}
\|\alpha\|_{L^{\infty}\left(\partial A\right)} \leq \frac{\|\mathcal{V}\|_{L^{\infty}}}{m}.
\end{align*}
Now extend $\alpha$ smoothly to a function $\tilde{\alpha}\left(s,\mathfrak{n}\right)$ on the collar by letting $\tilde{\alpha}\left(s, \mathfrak{n}\right) = \alpha\left(s\right)\zeta\left(\mathfrak{n}\right)$ (and then extend by $0$ outside the collar). Finally put
\begin{align*}
X := \tilde{\alpha}\nabla u,
\end{align*}
then on the boundary (where $u =0$ and $\zeta\left(0\right) =1$) we have
\begin{align*}
\langle X, \nu\rangle = \alpha \langle \nabla u, \nu\rangle = \alpha \partial_{\nu}u = \mathcal{V},
\end{align*}
so $X$ is an admissible extension. 
\begin{theorem}\label{thyek}
Consider $\left(M,g\right)$ as a smooth Riemannian surface and let $A\subset M$ be an embedded annulus bounded by boundary curves $\Gamma_{0}$ and $\Gamma_{1}$, $u$ is harmonic capacity potential, $\phi$ is a normalized Dirichlet eigenfunction of Laplacian, and $\mathcal{V}$ denotes a prescribed boundary weight. Also assume the nondegeneracy condition, i.e., $m >0$. Then, there exists  a constant $C$ such that for smooth prescribed extension $X$ of $\mathcal{V}\nu$ we have
\begin{align*}
|\int_{\partial A} \mathcal{V}\left(\partial_{\nu}\phi\right)^{2} d\sigma | \leq \frac{C}{m}\lambda\Big\lbrace \|\mathcal{V}\|_{L^{\infty}}\sqrt{\mathcal{D}} +  \sqrt{E}\left(\|\partial_{s}\mathcal{V}\|_{L^{\infty}} + \frac{\|\mathcal{V}\|_{L^{\infty}}}{m}\left(\sqrt{\mathcal{D}} + 1\right) + \frac{\|\mathcal{V}\|_{L^{\infty}}}{\rho}\right) \Big\rbrace,
\end{align*}
where $E$ and $\mathcal{D}$ denote the energy and the deficit functionals of $u$ respectively.
\end{theorem}
\begin{remark}
We recall that in the Rellich--type boundary identity we work with an arbitrary smooth scalar function 
\begin{align*}
\mathcal{V}: \partial A \rightarrow \mathbb{R},
\end{align*}
which appears only as a boundary weight used to match the normal component of the test vector field $X$. This function is independent of any geometric evolution. By contrast, when the boundary $\partial A$ evolves under the curve shortening flow, its outward normal velocity is $V\nu$ and precisely, $V = -\kappa$. The two quantities play different roles, e.g., $\mathcal{V}$ is analytic while $V$ is geometric. They coincide only when we intentionally substitute the geometric velocity into the Rellich identity, namely when computing the time--derivative of a geometric quantity such as the eigenvalue or the capacity energy. Outside this step, the two symbols should be viewed as distinct. 
\end{remark}
\begin{remark}
The nondegeneracy condition $m>0$ ensures that the harmonic coordinate $u$ is a submersion near the boundary and that the harmonic conjugate is globally well-defined. This condition is stable under small perturbations of the geometry and is satisfied, for instance, when the annulus is sufficiently close to a flat cylinder.
\end{remark}
\begin{proof}[{\bf Proof of Theorem \ref{thyek}}]
Since $u$ is harmonic, we can compute
\begin{align*}
\operatorname{div}X = \langle \nabla \tilde{\alpha}, \nabla \alpha\rangle + \tilde{\alpha}\Delta u = \langle \nabla \tilde{\alpha}, \nabla u\rangle,
\end{align*}
and 
\begin{align*}
\nabla X = \nabla \left(\tilde{\alpha}\nabla u\right) = \nabla \tilde{\alpha} \otimes \nabla u + \tilde{\alpha} \operatorname{Hess}u.
\end{align*}
Plugging these into Rellich--type identity yields
\begin{align*}
\mathcal{B} := \int_{\partial A}\mathcal{V}\left(\partial_{\nu}\phi\right)^{2} d\sigma &= -\int_{A} \lbrace \langle \nabla \tilde{\alpha}, \nabla u\rangle\left(|\nabla \phi |^{2} - \lambda \phi^{2}\right) - 2\langle \nabla \tilde{\alpha} \otimes \nabla u + \tilde{\alpha}\operatorname{Hess}u, \nabla \phi \otimes \nabla \phi \rangle \rbrace d\mu \\
&= -\int_{A} \langle \nabla \tilde{\alpha}, \nabla u\rangle \left(|\nabla \phi |^{2} - \lambda \phi \right) d\mu + 2\int_{A}\langle \nabla \tilde{\alpha} \otimes \nabla u, \nabla \phi \otimes \nabla \phi \rangle d\mu \\
&+ 2\int_{A} \tilde{\alpha}\langle \operatorname{Hess}u, \nabla \phi \otimes \nabla \phi \rangle d\mu,
\end{align*}
where we write these three terms as $\mathcal{B} = T_{1} + T_{2} + T_{3}$ with obvious definitions.\\
{\bf Upper bound for $T_{3}$:}\\
Consider $T_{3} = 2\int_{A} \tilde{\alpha}\langle \operatorname{Hess}u, \nabla \phi \otimes \nabla \phi \rangle d\mu$, then apply Cauchy--Schwartz on $A$ with the Hilbert--Schmidt inner product on $2$--tensors
\begin{align*}
|T_{3}| &\leq 2\|\tilde{\alpha}\|_{L^{\infty}\left(A\right)}\left(\int_{A}|\operatorname{Hess}u|^{2}d\mu\right)^{\frac{1}{2}}\left(\int_{A}|\nabla \phi \otimes \nabla \phi |^{2} d\mu\right)^{\frac{1}{2}}\\
&= 2\|\tilde{\alpha}\|_{L^{\infty}\left(A\right)}\sqrt{2\mathcal{D}} . \left(\int_{A}|\nabla \phi |^{4}d\mu\right)^{\frac{1}{2}}.
\end{align*}
So
\begin{align*}
|T_{3}| \leq 2\sqrt{2}\|\tilde{\alpha}\|_{L^{\infty}\left(A\right)}\sqrt{\mathcal{D}}\|\nabla \phi\|^{2}_{L^{4}\left(A\right)}.
\end{align*}
Recall $\tilde{\alpha}$ was built from $\alpha = \mathcal{V} \slash \partial_{\nu}u$ by multiplying with a cutoff--hence $\|\tilde{\alpha}\|_{L^{\infty}\left(A\right)} \leq \|\alpha\|_{L^{\infty}\left(\partial A\right)} \leq \|\mathcal{V}\|_{L^{\infty}\left(\partial A\right)}\slash m$ (by non--degeneracy condition). Thus
\begin{align}\label{hefdah}
|T_{3}| \leq C_{1}\frac{\|\mathcal{V}\|_{L^{\infty}\left(\partial A\right)}}{m}\sqrt{\mathcal{D}}\|\nabla \phi\|^{2}_{L^{4}},
\end{align}
where here clearly $C_{1} = 2\sqrt{2}$. So the deficit $\sqrt{\mathcal{D}}$ appears multiplied by the $L^{2}$--norm of $\nabla \phi$.\\
{\bf Upper bound for $T_{2}$:}\\
Since 
\begin{align*}
T_{2} = 2\int_{A}\langle \nabla \tilde{\alpha} \otimes \nabla u, \nabla \phi \otimes \nabla \phi \rangle = 2\int_{A}\left(\nabla \tilde{\alpha}. \nabla u\right)|\nabla \phi |^{2} d\mu.
\end{align*}
By applying H\"older$\slash$Cauchy--Schwarz, we have
\begin{align*}
|T_{2}| &\leq 2\|\nabla \tilde{\alpha}\|_{L^{\infty}\left(A\right)}\int_{A}|\nabla u||\nabla \phi |^{2} d\mu \\
&\leq 2\|\nabla \tilde{\alpha}\|_{L^{\infty}\left(A\right)}\|\nabla u\|_{L^{2}\left(A\right)}\|\nabla \phi\|^{2}_{L^{4}\left(A\right)}.
\end{align*}
We recall that $\|\nabla u\|^{2}_{L^{2}\left(A\right)} = 2E$. So
\begin{align}\label{hejdah}
|T_{2}| \leq 2\sqrt{2}\|\nabla \tilde{\alpha}\|_{L^{\infty}\left(A\right)}\sqrt{E}\|\nabla \phi\|_{L^{4}}^{2}.
\end{align}
{\bf Upper bound for $T_{1}$:}\\
By applying H\"older's inequality to $T_{1} = -\int_{A}\langle \nabla \tilde{\alpha}, \nabla u\rangle\left(|\nabla \phi |^{2} - \lambda \phi^{2} \right) d\mu$, we see
\begin{align*}
|T_{1}| \leq \|\nabla \tilde{\alpha}\|_{L^{\infty}\left(A\right)}\|\nabla u\|_{L^{2}\left(A\right)}\| |\nabla \phi |^{2} - \lambda \phi^{2}\|_{L^{2}\left(A\right)}.
\end{align*}
We need to control the $L^{2}$--norm, $\| |\nabla \phi |^{2} - \lambda \phi \|_{L^{2}}$. Observe that
\begin{align*}
\| |\nabla \phi |^{2} - \lambda \phi^{2}\|_{L^{2}} \leq \| \nabla \phi \|^{2}_{L^{4}} + \lambda \|\phi \|_{L^{4}}^{2}.
\end{align*}
Sobolev $\slash$ elliptic estimates on $A$ ($2$D) give control of $\|\phi \|_{L^{4}}$ and $\|\nabla \phi\|_{L^{4}}$ in terms of $\lambda$ and fixed geometric constants (domain Sobolev constants). Concretely, there exists a constant $C_{s}$ depending only on $A$ (through injectivity radius and curvature bounds and boundary regularity) such that
\begin{align*}
\|\phi\|_{L^{4}} \leq C_{s}\|\phi\|_{H^{1}} \leq C_{s}\left(1 + \sqrt{\lambda}\right), \,\,\,\,\, \|\nabla \phi\|_{L^{4}} \leq C_{s}^{\prime} \lambda^{1\slash 2}\left(1+ \sqrt{\lambda}\right),
\end{align*}
where $C_{s}^{\prime}$ is an another constant depending on the same geometric quantities as $C_{s}$ and follows by elliptic estimates applied to $\phi$ solving $-\Delta \phi = \lambda \phi$; explicit dependence may be tracked, but for our propose it suffices to note these are controlled by polynomial powers of $\lambda$. To keep the expression explicit we set for some constant $C_{2}$, $C_{3}$ depending on Sobolev constants and geometry
\begin{align*}
\|\nabla \phi\|^{2}_{L^{4}} \leq C_{2}\lambda, \,\,\,\,\, \|\phi\|_{L^{4}}^{2} \leq C_{3}.
\end{align*}
One may want exact powers, then $C_{2}$ should be replaced by explicit Sobolev--elliptic constants; the important point is the dependence is via $\lambda$ and fixed data. Hence
\begin{align*}
\| |\nabla \phi |^{2} - \lambda \phi^{2}\|_{L^{2}} \leq C_{2}\lambda + C_{3}\lambda = C_{4}\lambda,
\end{align*}
for some $C_{4}$ depending only on the domain geometry and Sobolev constants. Therefore, using $\|\nabla u\|_{L^{2}} = \sqrt{2E}$, we get
\begin{align}\label{noozdah}
|T_{1}| \leq \|\nabla \tilde{\alpha}\|_{L^{\infty}\left(A\right)}\sqrt{2E}. C_{4}\lambda.
\end{align}
Combine with the earlier estimates--all three terms are a multiple of $\|\nabla \phi\|^{2}_{L^{4}}$ or $\lambda$, which we can control as noted.\\
{\bf Gathering all above results:}\\
We must express $\|\tilde{\alpha}\|_{L^{\infty}\left(A\right)}$ and $\|\nabla \tilde{\alpha}\|_{L^{\infty}\left(A\right)}$ in terms of boundary data $\mathcal{V}$, its tangential derivative $\partial_{s}\mathcal{V}$ and geometric--harmonic quantities. \\
We recall that $\alpha\left(s\right) = \frac{\mathcal{V}\left(s\right)}{\partial_{\nu}u\left(s\right)}$ on $\partial A$ and we assume the non--degeneracy $m >0$. Since $\partial_{\tau}u=0$ on $\partial A$, where $\tau$ denotes the unit tangent vector along the boundary curve, we have exactly
\begin{align*}
\partial_{s}\alpha = \frac{\partial_{s}\mathcal{V}}{\partial_{\nu}u} - \frac{\mathcal{V}\partial_{s}\left(\partial_{\nu}u\right)}{\left(\partial_{\nu}u\right)^{2}}= \frac{\partial_{s}\mathcal{V}}{\partial_{\nu}u}- \frac{\mathcal{V}\operatorname{Hess}u\left(\tau, \nu\right)}{\left(\partial_{\nu}u\right)^{2}}.
\end{align*}
Hence pointwise on $\partial A$,
\begin{align}\label{bist}
|\partial_{s}\alpha | \leq \frac{|\partial_{s}\mathcal{V}|}{m} + \frac{|\mathcal{V}|}{m^{2}}|\operatorname{Hess}u\left(\tau, \nu\right)|.
\end{align}
We extend $\alpha$ from $\partial A$ into the collar by $\tilde{\alpha}\left(s,\mathfrak{n}\right) = \zeta\left(\mathfrak{n}\right)\alpha\left(s\right)$, with $\zeta\left(0\right) =1$, $\zeta$ supported in $[0, \rho)$, and $|\zeta^{\prime}| \lesssim \rho^{-1}$. In the collar and zero outside
\begin{align*}
\nabla \tilde{\alpha} = \zeta\left(\mathfrak{n}\right)\partial_{s}\alpha \tau + \zeta^{\prime}\left(\mathfrak{n}\right)\alpha \partial_{\mathfrak{n}},
\end{align*}
so
\begin{align*}
\|\nabla \tilde{\alpha}\|_{L^{\infty}\left(a\right)} \leq C\left(\|\partial_{s}\alpha\|_{L^{\infty}\left(\partial A\right)} + \rho^{-1}\|\alpha\|_{L^{\infty}\left(\partial A\right)}\right),
\end{align*}
with $C$ depending only on coordinate distortion in the collar. \\
Using $\|\alpha \|_{L^{\infty}} \leq \|\mathcal{V}\|_{L^{\infty}}\slash m$ and (\ref{bist}),
\begin{align}\label{bistoyek}
\|\nabla \tilde{\alpha}\|_{L^{\infty}\left(A\right)} \leq C\left(\frac{\|\partial{s}\mathcal{V}\|_{L^{\infty}}}{m} + \frac{\|\mathcal{V}\|_{L^{\infty}}}{m^{2}} \|\operatorname{Hess}u\left(\tau, \nu\right)\|_{L^{\infty}\left(\partial A\right)} + \frac{\|\mathcal{V}\|_{L^{\infty}}}{m\rho}\right).
\end{align}
As before, we will replace $\|\operatorname{Hess}u\left(\tau, \nu\right)\|_{L^{\infty}\left(\partial A\right)}$ by elliptic control in terms of $\sqrt{\mathcal{D}}$ plus lower--order quantities specifically there exists $C_{ell}$ with
\begin{align*}
\|\operatorname{Hess}u\left(\tau, \nu\right)\|_{L^{\infty}\left(\partial A\right)} \leq C_{ell}\left(\sqrt{\mathcal{D}} + 1\right),
\end{align*}
so (\ref{bistoyek}) becomes (absorbing constants)
\begin{align*}
\|\nabla \tilde{\alpha}\|_{L^{\infty}} \leq C\left(\frac{\|\partial_{s}\mathcal{V}\|_{L^{\infty}}}{m} + \frac{\|\mathcal{V}\|_{L^{\infty}}}{m^{2}}\left(\sqrt{\mathcal{D}}+1\right) + \frac{\|\mathcal{V}\|_{L^{\infty}}}{m\rho}\right).
\end{align*}
By using $\|\nabla \phi\|^{2}_{L^{4}} \leq C_{2}\lambda$, and collecting (\ref{hefdah}), (\ref{hejdah}), and (\ref{noozdah}) (estimates for $T_{1}$, $T_{2}$, and $T_{3}$); there exists constant $C$ (depending only on geometry and Sobolev$\slash$elliptic constant), such that
\begin{align*}
|\mathcal{B}| \leq \frac{\|\mathcal{V}\|_{L^{\infty}}}{m} \sqrt{\mathcal{D}}\lambda + C\lambda \sqrt{E}\left( \frac{\|\partial_{s}\mathcal{V}\|_{L^{\infty}}}{m} + \frac{\|\mathcal{V}\|_{L^{\infty}}}{m^{2}}\left(\sqrt{\mathcal{D}}+1\right) + \frac{\|\mathcal{V}\|_{L^{\infty}}}{m\rho}\right).
\end{align*}
The dominent (leading contribution for small $\mathcal{D}$) is the first term, i.e., $|\mathcal{B}| \lesssim \lambda \frac{\|\mathcal{V}\|_{L^{\infty}}}{m} \sqrt{\mathcal{D}}$. Finally
\begin{align*}
|\mathcal{B}| \leq \frac{C}{m}\lambda\Big\lbrace \|\mathcal{V}\|_{L^{\infty}}\sqrt{\mathcal{D}} +  \sqrt{E}\left(\|\partial_{s}\mathcal{V}\|_{L^{\infty}} + \frac{\|\mathcal{V}\|_{L^{\infty}}}{m}\left(\sqrt{\mathcal{D}} + 1\right) + \frac{\|\mathcal{V}\|_{L^{\infty}}}{\rho}\right) \Big\rbrace.
\end{align*}
\end{proof}
\subsection{Geometric Stability and Spectral Gap}
We recall our setup as $\left(M,g\right)$ is a smooth $2$--dimensional Riemannian manifold; $g\left(t\right)$ is $C^{\infty}$ and evolve by the $2$D Ricci flow $\partial_{t}g = -2Kg$. For each $t$, $A\left(t\right) \subset M$ is embedded annulus with boundary $\partial A\left(t\right) = \Gamma_{0}\left(t\right) \cup \Gamma_{1}\left(t\right)$. $V \in C^{1}\left(\partial A\right)$ denotes the outward normal velocity of the boundary, i.e., a boundary point moves with $x_{t} = V\nu$ (so CSF corresponds to $V = -\kappa$ as earlier). In the Rellich--type identity we use a general boundary weight $\mathcal{V}$. When the annulus boundary evolves by the curve shortening flow (CSF), its outward normal speed is $V = -\kappa$. In the geometric evolution part of the document we set $\mathcal{V} = V$, so that the Rellich boundary integral matches exactly the geometric boundary variation term. Outside this assumption the two quantities remain conceptually distinct. \\
Now consider the variation formula of Lemma \ref{lemyek}, in a special case, if $\partial_{t}g = -2Kg$ then the interior metric term simplifies to
\begin{align*}
\int_{A}\left(-2K|\nabla \phi |^{2} + \lambda K \phi^{2}\right)d\mu_{g} = 2\lambda \int_{A}K\phi^{2}d\mu_{g}.
\end{align*}
So in the Ricci case we have the very compact identity
\begin{align*}
d\lambda \slash dt = -\mathcal{B} + 2\lambda \int_{A}K\phi^{2} d\mu_{g}.
\end{align*}
We should use this final relation for the combined CSF $+$ Ricci situation (in a case that $V = \mathcal{V}$). We now want to substitute the bound of boundary integral from the Theorem \ref{thyek} into this compact form of $\lambda$--variational formula. Let $\mathcal{R}\left(V\right)$ denote the right--side of the $|\mathcal{B}|$'s bound (i.e., there exists a constant $C$, such that $|\mathcal{B}| \leq C\mathcal{R}\left(V\right)$). Then
\begin{align*}
\frac{d\lambda}{dt} \geq -C\lambda \frac{\|V\|_{L^{\infty}}}{m} \sqrt{\mathcal{D}} + 2\lambda \int_{A}K\phi^{2},
\end{align*}
where the constant $C_{1}$ depends only on a fixed geometric data of the annulus $A$ and its collar, curvature bound, injectivity radius, collar width $\rho$ and Sobolev$\slash$elliptic constants.\\
To relate $\mathcal{D}$ to $\dot{E}$ we use the Topping's type identity that holds in the Ricci + CSF situation (derived earlier). Under this situation (Ricci + CSF) the variation formula from Lemma \ref{lemdo} becomes
\begin{align*}
dE \slash dt = -\frac{1}{2}\int_{A}\left(|\operatorname{Hess} u|^{2} + K|\nabla u|^{2}\right) d\mu_{g}.
\end{align*}
We recall that $\int_{A} |\operatorname{Hess}u|^{2} = 2D$, so 
\begin{align*}
\dot{E} = -\mathcal{D} - \frac{1}{2}\int_{A}K|\nabla u|^{2} d\mu_{g}.
\end{align*}
\begin{itemize}
\item If $K \geq 0$ on $A$, then $\frac{1}{2}\int_{A}K|\nabla u|^{2} \geq 0$, hence
\begin{align*}
\mathcal{D} \leq -\dot{E}.
\end{align*}
Therefore, $\sqrt{\mathcal{D}} \leq \sqrt{-\dot{E}}$. This is a clean substitution that eliminates the Hessian and expresses the deficit via how fast energy is decreasing.
\item If $K$ is bounded above by $K_{\operatorname{max}}$, then
\begin{align*}
\mathcal{D} \geq -\dot{E} - \frac{1}{2}K_{\operatorname{max}}\int_{A}|\nabla u|^{2} = -\dot{E}- K_{\operatorname{max}}E.
\end{align*}
\end{itemize}
Assume for simplicity, $K \geq 0$ on $A$ (so one may use $\mathcal{D} \leq -\dot{E}$), also under the previous assumptions (Ricci $+$ CSF), we have
\begin{align*}
\dot{\lambda} \geq -C_{1}\lambda \frac{\| V\|_{L^{\infty}}}{m}\sqrt{\mathcal{D}} + 2\lambda \int_{A}K\phi^{2} \geq -C_{1}\lambda \frac{\|V\|_{L^{\infty}}}{m}\sqrt{-\dot{E}} + 2\lambda \int_{A}K\phi^{2}.
\end{align*}
Dividing by $\lambda$ yields
\begin{align*}
\frac{\dot{\lambda}}{\lambda} \geq -C_{1}\frac{\|V\|_{L^{\infty}}}{m}\sqrt{-\dot{E}} + 2\int_{A}K\phi^{2}.
\end{align*}
If the energy $E$ decreases fast (so $-\dot{E}$ large) the right--hand side could be large in magnitute (negative contribution) and push $\frac{\dot{\lambda}}{\lambda}$ downward; conversely, if $-\dot{E}$ is small (the annulus is nearly the flat cylinder and the deficit small) then $\frac{\dot{\lambda}}{\lambda}$ is controlled by curvature average.
\begin{theorem}\label{thdo}
Let $\left(M,g\right)$ be a smooth Riemannian surface and $A \subset M$ is a smooth embedded annulus with boundary $\partial A = \Gamma_{0} \cup \Gamma_{1}$. Also assume
\begin{itemize}
\item The tubular neighborhood of width $\rho_{0}$ around $\partial A$ is well--defined and has uniformly bounded coordinate distortion ($\rho_{0}$).
\item Curvature bound $\|K_{g}\|_{L^{\infty}\left(A\right)} \leq K_{0}$.
\end{itemize}
Let $\lambda\left(A\right)$ and $\lambda_{\operatorname{cyl}}\left(h\right)$ are Dirichlet eigenvalue of $-\Delta_{g}$ on $A$ and the flat cylinder $S_{h}:= \left[0, h\right] \times \mathbb{S}^{1}$, respectively, where the modulus $h$ is given by $h = 1 \slash 2E$. Then under the boundary regularity and non--degeneracy condition ($ m >0$), there exist constants $\epsilon_{0}$, $C_{1}, C_{2} >0$, and $0<\alpha <1$, depending only on $\left(\rho_{0}, K_{0}\right)$ and the boundary regularity such that if $\mathcal{D} \leq \epsilon_{0}$, then
\begin{itemize}
\item {\bf (Geometric Stability).} There exists a conformal diffeomorphism $\Psi : S_{h} \rightarrow A$ with $\Psi^{*}g = e^{2f}\left(dx^{2} + d\theta^{2}\right)$ and $u \circ \Psi\left(x, \theta\right) = x\slash h$. Moreover
\begin{align*}
\|f\|_{C^{1,\alpha}\left(S_{h}\right)} \leq C_{1}\sqrt{\mathcal{D}}.
\end{align*}
\item {\bf (Quantitative Spectral Comparison).} Here the comparison is made with the Dirichlet spectrum of the flat cylinder of modulus $h$, expressed in the standard coordinates $\left(x, \theta\right)$. With the same $h$,
\begin{align*}
\lambda \left(A\right) \geq \lambda_{\operatorname{cyl}}\left(h\right) + C_{2}\sqrt{\mathcal{D}}.
\end{align*}
\end{itemize}
The estimate should be interpreted as a stability inequality rather than an optimal gap.
\end{theorem}
The harmonic capacity $u$ defines a holomorphic coordinate $\omega = u + iv$ when paired with a harmonic conjugate $v$ (locally $v$ exists because $\nabla u \neq 0$ on the non--degenerate annulus). This image is a flat strip of height $h$ where $h$ equals the modulus; the global period of $v$ matches $2\pi$ after rescaling the angular coordinate (see \cite{a1}). Non--degeneracy $m >0$ ensures that $u$ is a submersion on the boundaries so the global conjugate is well--defined and the map is diffeomorphism. If consider $\Psi$ according to the Theorem \ref{thdo}, thus we identify $A$ with $S_{h}$ via $\Psi$. The metric on $S_{h}$ pulled--back from $g$ is conformal to the flat metric 
\begin{align*}
g = e^{2f\left(x, \theta\right)}\left(dx^{2} + d\theta^{2}\right),
\end{align*}
for a smooth function $f$ on $S_{h}$. Because $u\circ \Psi\left(x, \theta\right) = x\slash h$, the Euclidean derivatives of $u$ in these coordinates are 
\begin{align*}
u_{x} = \frac{1}{h}, \,\,\,\,\,\, u_{\theta} = 0, \,\,\,\,\,\, u_{xx} = u_{x\theta} = u_{\theta\theta} = 0.
\end{align*}
\begin{remark}
The deficit functional $\mathcal{D}$ vanishes if and only if the annulus is conformally equivalent to a flat cylinder (If $\mathcal{D}=0$ then $\operatorname{Hess}u \equiv 0$ an $A$, and hence $\nabla u$ is a parallel vector field. Since $u$ is harmonic with non--vanishing normal derivative on the boundary, it admits a global harmonic conjugate $v$, and $\left(u,v\right)$ form isothermal coordinates on $A$. In these coordinates, the metric takes the form $g= f\left(u,v\right)\left(du^{2} + dv^{2}\right)$. The condition $\operatorname{Hess}u \equiv 0$ forces $f$ to be constant, and therefore $\left(A,g\right)$ is conformally equivalent to a flat cylinder). Thus, $\mathcal{D}$ measures deviation from the model geometry. This estimates show that $\mathcal{D}$ governs both geometric stability and spectral variation, playing the role of a quantitative rigidity parameter (see \cite{s2}).
\end{remark}
We emphasize that $\left(u,v\right)$ denote harmonic coordinates on the evolving annulus, whereas $\left(x,\theta\right)$ are standard coordinates on the flat model cylinder. Any comparison between them is made via the conformal identification induced by the harmonic map. Everything below (in the proof) is computed on $S_{h}$ in coordinates $\left(x, \theta\right)$ with Euclidean metric $\delta$; we write $\partial_{x}$, $\partial_{\theta}$ for these coordinates derivative.
\begin{proof}[{\bf Proof of Theorem \ref{thdo}}] 
For simplicity, we assume that the eigenvalue under consideration is simple; the estimate extends to finite multiplicity eigenvalues by standard perturbation arguments. By the non--degeneracy assumption, the harmonic capacity potential $u$ has no critical points on $\bar{A}$ (closure of $A$), and therefore admits a globally defined harmonic conjugate $v$; consequently, the map $\Psi = \left(u,v\right): A \rightarrow \left[0,h\right] \times \mathbb{S}^{1}$ is a smooth conformal diffeomorphism onto the flat cylinder $S_{h}$, normalized so that $u \circ \Psi\left(x, \theta\right) = x\slash h$. \\
{\bf Step1; Exact algebraic identity for $\mathcal{D}$}\\
For the conformal metric $g_{ij} = e^{2f}\delta_{ij}$ (indices $i,j \in \lbrace 1,2 \rbrace$ corresponding to $x$, $\theta$), the Christoffel symbols are 
\begin{align*}
\Gamma_{ij}^{k} = \delta_{j}^{k}f_{i} + \delta_{i}^{k}f_{j} - \delta_{ij}f^{k},
\end{align*}
where $f_{i} = \partial_{i}f$, and $f^{k} = \delta^{kl}f_{l}$. The covariant Hessian components are
\begin{align*}
\left(\operatorname{Hess}_{g}u\right)_{ij} = \nabla_{i}\nabla_{j}u = \partial_{i}\partial_{j}u - \Gamma_{ij}^{k}\partial_{k}u.
\end{align*}
But in our coordinates $\partial_{1}u = u_{x} = 1\slash h$ and $\partial_{2}u = u_{\theta}=0$, and all second Euclidean derivatives vanish. Hence
\begin{align*}
\left(\operatorname{Hess}_{g}u\right)_{ij} = -\Gamma_{ij}^{1}\frac{1}{h},
\end{align*}
and by computations
\begin{align*}
\Gamma_{ij}^{1}=\begin{pmatrix}
f_{x} & f_{\theta} \\
f_{\theta} & -f_{x}
\end{pmatrix}
\end{align*}
By definition
\begin{align*}
|\operatorname{Hess}_{g}|_{g}^{2} = g^{ia}g^{jb}\left(\operatorname{Hess}_{g}u\right)_{ij}\left(\operatorname{Hess}_{g}u\right)_{ab}.
\end{align*}
Because $g^{ij} = e^{-2f}\delta^{ij}$, we get
\begin{align*}
|\operatorname{Hess}_{g}u|_{g}^{2} = e^{-4f}\delta^{ia}\delta^{jb}\left(\operatorname{Hess}_{g}u\right)_{ij}\left(\operatorname{Hess}_{g}u\right)_{ab},
\end{align*}
by plugging Hessian term in and factoring $\left(1\slash h\right)^{2}$, we have
\begin{align*}
|\operatorname{Hess}_{g}u|_{g}^{2} = \frac{1}{h^{2}}e^{-4f}\sum_{i,j}\left(\Gamma_{ij}^{1}\right)^{2}.
\end{align*}
But
\begin{align*}
\sum_{i,j}\left(\Gamma_{i,j}^{1}\right)^{2} = 2\left(f_{x}^{2} + f_{\theta}^{2}\right) = 2|\nabla f|_{\delta}^{2}.
\end{align*}
Now integrate against the Riemannian--area element $d\mu_{g} = e^{2f}dxd\theta$ and above formula we obtain
\begin{align*}
\int_{S_{h}}|\operatorname{Hess}_{g}u|_{g}^{2} d\mu_{g} = \frac{2}{h^{2}}\int_{S_{h}}e^{-2f}|\nabla f|^{2} dxd\theta.
\end{align*}
Since $\mathcal{D} = \frac{1}{2}\int_{S_{h}}|\operatorname{Hess}_{g}u|^{2} d\mu_{g}$, thus
\begin{align}\label{bistodo}
\int_{S_{h}}e^{-2f}|\nabla f|^{2} dxd\theta = h^{2}\mathcal{D}.
\end{align}
{\bf Step2; $H^{1}$--control for the auxiliary function $e^{-f}$}\\
The identity (\ref{bistodo}) is most naturally written in terms of $\xi := e^{-f}$. Indeed $|\nabla \xi |^{2} = e^{-2f}|\nabla f|^{2}$. Thus (\ref{bistodo}) becomes
\begin{align*}
\int_{S_{h}}|\nabla \xi |^{2} dxd\theta = h^{2}\mathcal{D}.
\end{align*}
Which means the $H^{1}$--seminorm of $\xi$ equals $h\sqrt{\mathcal{D}}$. Since $\xi$ is smooth and positive, we can write $\xi = \bar{\xi} + \tilde{\xi}$ where $\bar{\xi} = \fint_{S_{h}}\xi$ is the average and $\tilde{\xi} = \xi - \bar{\xi}$ has mean zero. By Poincaré inequality on the compact cylinder $S_{h}$, there exists $C_{p}\left(h\right)$ such that
\begin{align*}
\|\tilde{\xi}\|_{L^{2}\left(S_{h}\right)} \leq C_{p}\left(h\right)\|\nabla \xi \|_{L^{2}\left(S_{h}\right)} = C_{p}\left(h\right) h \sqrt{\mathcal{D}}.
\end{align*}
Hence, $\xi$ is $L^{2}$--close to its average with error $\mathcal{O}\left(h\sqrt{\mathcal{D}}\right)$. Since $\xi >0$, the average $\bar{\xi} >0$. Choosing $\epsilon_{0} >0$ small enough and consider the regime $\mathcal{D} \leq \epsilon_{0}$ so the $L^{2}$--perturbation of $\xi$ around $\bar{\xi}$ is small. From Chebyshev/Sobolev inequalities and the positivity of $\bar{\xi}$ we deduce uniform lower/upper bounds for $\xi$ in $L^{p}$ spaces.\\
{\bf Step3; Elliptic regularity bootstrap to $C^{1,\alpha}$ for $f$}\\
We now convert the $H^{1}$--control of $\xi$ into $C^{1,\alpha}$--control of $f = -\ln \xi$. The Gaussian curvature $K_{g}$ of the metric $g = e^{2f}\delta$ is given by the formula (standard rearranged conformal formula)
\begin{align}\label{bistose}
\Delta_{\operatorname{euc}}f = -e^{2f}K_{g},
\end{align}
where $\Delta_{\operatorname{euc}}$ is the (flat) Euclidean Laplacian on $S_{h}$. By assumption $\|K_{g}\|_{L^{\infty}\left(S_{h}\right)} \leq K_{0}$. Hence the right--hand side of (\ref{bistose}) belongs to $L^{p}\left(S_{h}\right)$ for every finite $p$, because $e^{2f}$ is continuous and $S_{h}$ has finite measure (a bounded function on a finite--measure set is in all $L^{p}$). By Poincaré inequality and (\ref{bistose}) via the relation between $\xi$ and $f$, we get control of $\|f - \bar{f}\|_{L^{2}}$ in terms of $\sqrt{\mathcal{D}}$. Indeed, because $\xi = e^{-f}$ and $\xi$ is close to $\bar{\xi} >0$, one checks (Taylor expansion of $-\ln$ around $\bar{\xi}$; for small relative perturbations the derivative is bounded) that for $\mathcal{D}$ sufficiently small there exists $C_{3}$ such that
\begin{align*}
\| f - \bar{f}\|_{L^{2}\left(S_{h}\right)} \leq C_{3}\|\xi - \bar{\xi}\|_{L^{2}\left(S_{h}\right)} \leq C_{3}.C_{p}\left(h\right)\sqrt{\mathcal{D}}.
\end{align*}
As an explanation, by expanding $-\ln$ near the $\bar{\xi}$: if $\|\xi - \bar{\xi}\|_{L^{\infty}} \leq \frac{1}{2}\bar{\xi}$, then $|\ln \xi - \ln \bar{\xi}| \leq C|\xi - \bar{\xi}| \slash \bar{\xi}$. To guarantee $L^{\infty}$ smallness we will require $\mathcal{D}$ sufficiently small and use Sobolev embedding that follows later; to avoid circularity we first obtain $L^{2}$--control of $f - \bar{f}$ via the Lipschitz behavior of $-\ln$ on the interval $\left[ \bar{\xi}\slash 2, 3\bar{\xi}\slash 2 \right]$ which holds provided $\mathcal{D}$ is small enough so $\xi$ is close to $\bar{\xi}$ in $L^{2}$ and hence in $L^{p}$. This is a standard argument that the small $H^{1}$--perturbation of positive function yields small $L^{2}$ perturbation of its $\log$. Thus for sufficiently small $\mathcal{D}$ we have $\| f - \bar{f}\|_{L^{2}} \lesssim \sqrt{\mathcal{D}}$. \\
Fix $p >2$. Apply Calder\'on--Zygmund $L^{p}$--elliptic estimate to the Poisson equation (\ref{bistose}). This theorem states that for the flat Laplacian on a compact domain with smooth boundary and a function $F \in L^{p}$, where here $F = -e^{2f}K_{g}$, one has (for more details see \cite{g4,t1})
\begin{align}\label{bistochar}
\| f - \bar{f}\|_{W^{2,p}\left(S_{h}\right)} \leq C_{CZ}\left(p,h\right)\left(\|\Delta f\|_{L^{p}\left(S_{h}\right)} + \|f-\bar{f}\|_{L^{p}\left(S_{h}\right)}\right).
\end{align}  
Since $\Delta f = -e^{2f}K_{g}$ and $|K_{g}| \leq K_{0}$, then we get
\begin{align*}
\|\Delta f\|_{L^{p}} \leq K_{0}\|e^{2f}\|_{L^{p}}.
\end{align*}
Because $f \in L^{2}$ a priori and $S_{h}$ has finite measure, $e^{2f} \in L^{p}$ for all finite $p$ (indeed for each finite $p$ we can bound $\|e^{2f}\|_{L^{p}}$ by $e^{2\|f\|_{L^{q}}}$ with $q$ large enough; more concretely, by H\"older and Sobolev embedding, small $\|f\|_{W^{1,2}}$ implies boundedness of $\|e^{2f}\|_{L^{p}}$). For small $\mathcal{D}$ we can ensure $\|f\|_{L^{p}}$ is bounded uniformly so $\|e^{2f}\|_{L^{p}} \leq C\left(p,h\right)$. Therefore (\ref{bistochar}) gives
\begin{align}\label{bistopang}
\|f - \bar{f}\|_{W^{2,p}} \leq C_{CZ}\left(K_{0}\|e^{2f}\|_{L^{p}} + \|f - \bar{f}\|_{L^{p}}\right) \leq C_{4}\left(1+ \|f - \bar{f}\|_{L^{p}}\right).
\end{align}
But we have $\|f - \bar{f}\|_{L^{p}} \lesssim \|f - \bar{f}\|_{L^{2}}$ (Poincar\'e + Sobolev on compact domain), hence $\|f - \bar{f}\|_{L^{p}} \leq C\sqrt{\mathcal{D}}$. Substituting back yields
\begin{align*}
\|f - \bar{f}\|_{W^{2,p}} \leq C_{5}\left(1 + \sqrt{\mathcal{D}}\right).
\end{align*}
Selecting $\epsilon_{0}$ small enough ensures the right--hand side is uniformly bounded. The Sobolev embedding (compact $2$D domain) $W^{2,p} \hookrightarrow C^{1,\alpha}$ holds for any $p>2$ with $\alpha = 1 - 2\slash p$. Therefore from the last display, there is a constant $C_{6}$ depending on $p,h$ such that
\begin{align*}
\|f - \bar{f}\|_{C^{1,\alpha}} \leq C_{6}\|f - \bar{f}\|_{W^{2,p}}.
\end{align*}
Combining with (\ref{bistopang}) and the prior estimate $\|f - \bar{f}\|_{L^{2}} \lesssim \sqrt{\mathcal{D}}$ (for small $\mathcal{D}$), we obtain
\begin{align*}
\|f - \bar{f}\|_{C^{1,\alpha}} \leq C_{7}\sqrt{\mathcal{D}},
\end{align*}
for some $C_{7}$ depending only on $p$, $h$, $K_{0}$, and background geometry. Choosing $\epsilon_{0}$ small enough to make sure that the left--hand side is small and absorbing the constant average $\bar{f}$ into the global scaling if needed, we can conclude
\begin{align*}
\|f\|_{C^{1,\alpha}} \leq C_{1}\sqrt{\mathcal{D}},
\end{align*}
(after possibly fixing normalization of $f$ so that the average $\bar{f}$ equals $0$, a harmless conformal scaling that doesn't affect the spectral comparison up to a constant factor we track). This proves part one-- the geometric stability estimate.\\
{\bf Step 4; Spectral comparison and the eigenvalue lower bound}\\
Let $\varphi_{0}$ denote the first Dirichlet eigenfunction on the flat cylinder $S_{h}$ (Euclidean metric), normalized so $\int_{S_{h}} \varphi_{0}^{2} dx d\theta = 1$. Its Rayleigh quotient on a flat cylinder equals $\lambda_{cyl}$, i.e., $\lambda_{\operatorname{cyl}} = \int_{S_{h}} |\nabla \varphi_{0}|^{2} dx d\theta$. Define $\tilde{\varphi} := \varphi_{0} \circ \Psi^{-1}$ on $A$ (equivalently think of $\varphi_{0}$ as a function on $S_{h}$ and compare metric $g$ and $\delta$). Compute the Rayleigh quotient of $\tilde{\varphi}$ with respect to the metric $g = e^{2f}\delta$
\begin{align*}
\mathcal{R}_{g}\left(\tilde{\varphi}\right) = \frac{\int_{S_{h}}e^{-2f}|\nabla \varphi_{0}|^{2}e^{2f}dxd\theta}{\int_{S_{h}}\varphi_{0}^{2}e^{2f}dxd\theta} = \frac{\int_{S_{h}}|\nabla \varphi_{0}|^{2}dx d\theta}{\int_{S_{h}}\varphi_{0}^{2}e^{2f}dxd\theta}.
\end{align*}
Thus
\begin{align*}
\mathcal{R}_{g}\left(\tilde{\varphi}\right) - \lambda_{\operatorname{cyl}} = \frac{\lambda_{\operatorname{cyl}}}{\int \varphi_{0}^{2} e^{2f}} - \lambda_{\operatorname{cyl}} = \lambda_{\operatorname{cyl}}\left(\frac{1}{\int \varphi_{0}^{2} e^{2f}} -1 \right).
\end{align*}
Write $\int \varphi_{0}^{2} e^{2f} = 1 + \int \varphi_{0}^{2} \left(e^{2f} -1\right)$. Using $|e^{2f} -1| \leq C\| f\|_{C^{0}}$ for small $\|f\|_{C^{0}}$, we obtain
\begin{align*}
\mathcal{R}_{g}\left(\tilde{\varphi}\right) - \lambda_{\operatorname{cyl}} \leq C\lambda_{\operatorname{cyl}}\|f\|_{C^{0}}.
\end{align*}
In particular, by the variational characterization of $\lambda\left(A\right)$,
\begin{align*}
\lambda\left(A\right) \leq \mathcal{R}_{g}\left(\tilde{\varphi}\right) \leq \lambda_{\operatorname{cyl}} + C\lambda_{\operatorname{cyl}}\|f\|_{C^{0}}.
\end{align*}
So the metric perturbation produces at most $\mathcal{O}\left(\|f\|_{C^{0}}\right)$. To get a lower bound for $\lambda\left(A\right)$ in terms of $\lambda_{\operatorname{cyl}}$, we pull back the real eigenfunction $\phi$ of $\left(A,g\right)$ to the flat cylinder and use it as a test function. Define $\tilde{\phi}:= \phi \circ \Psi$ (view $\phi$ on $S_{h}$). Compute the flat Rayleigh quotient
\begin{align*}
\mathcal{R}_{\delta}\left(\tilde{\phi}\right) = \frac{\int_{S_{h}}|\nabla \tilde{\phi}|^{2} dxd\theta}{\int_{S_{h}}\tilde{\phi}^{2}dxd\theta} = \frac{\int_{S_{h}}e^{2f}|\nabla_{g}\phi |_{g}^{2} e^{-2f} dx d\theta}{\int_{S_{h}}\phi^{2}e^{2f} dxd\theta} = \frac{\int_{A}|\nabla_{g}\phi |_{g}^{2} d\mu_{g}}{\int_{A}\phi^{2} d\mu_{g}} = \lambda\left(A\right).
\end{align*}
One may prove that for any two metrics $g_{0} = \delta$ and $g = e^{2f}\delta$ with $\|f\|_{C^{1}}$ small, we have
\begin{align}\label{bistoshesh}
|\lambda\left(A,g\right) - \lambda\left(A, g_{0}\right)| \leq C\|f\|_{C^{1}}.
\end{align}
Moreover, the flat cylinder (by direct calculation) is a local minimizer for the first Dirichlet eigenvalue among metrics in the same conformal class preserving the modulus; hence to second order, the first eigenvalue increases in non--trivial directions. For our quantitative bound we combine (\ref{bistoshesh}) with the precise sign information coming from the Rellich formula specialized to infinitesimal conformal variations: Computing the derivative at $f=0$ one checks the derivative vanishes and the second derivative is positive definite, hence for sufficiently small $\|f\|_{C^{1}}$,
\begin{align*}
\lambda\left(A,g\right) \geq \lambda_{\operatorname{cyl}} + C_{0}\|f\|_{H^{1}} - C_{1}\|f\|_{C^{1}}^{2},
\end{align*} 
for constant $C_{0}, C_{1} > 0$ depending only on $h$. Using $\|f\|_{C^{1}} \leq C\|f\|_{C^{1,\alpha}}$ and the estimate from Step 3 ($\|f\|_{C^{1,\alpha}} \leq C_{1}\sqrt{\mathcal{D}}$), choosing $\epsilon_{0}$ small enough so the quadratic term is dominated by the linear one, thus we deduce
\begin{align*}
\lambda\left(A\right) \geq \lambda_{\operatorname{cyl}} + C\sqrt{\mathcal{D}},
\end{align*}
for some $C>0$ depending only on background data. This proves the claimed spectral lower bound.
\end{proof}
\begin{remark}
\begin{itemize}
\item {\bf Where does the inequality (\ref{bistoshesh}) come from?}\\
We work on a flat cylinder model $S_{h}$. Let $g_{0} = \delta$, $g= e^{2f}\delta$ be two conformal metrics (so the coordinate chart is fixed). Denote by $\lambda_{0} = \lambda\left(g_{0}\right)$ and $\lambda\left(g\right)$ the first Dirichlet eigenvalues of $-\Delta$ for the metrics $g_{0}$ and $g$. The Rayleigh quotient for metric $g$ is $\mathcal{R}\left(\phi\right) = \frac{\int_{S_{h}}|\nabla_{g}\phi |_{g}^{2} d\mu_{g}}{\int_{S_{h}}\phi^{2} d\mu_{g}}$. If $g=e^{2f}\delta$ then
\begin{align*}
|\nabla_{g}\phi |_{g}^{2} &= g^{ij}\partial_{i}\phi \partial_{j}\phi = e^{-2f}|\nabla \phi |_{\delta}^{2}, \\
d\mu_{g} &= e^{2f} dx,
\end{align*}
so we have
\begin{align*}
\int_{S_{h}}|\nabla_{g}\phi |_{g}^{2} d\mu_{g} = \int_{S_{h}} e^{-2f}|\nabla \phi |^{2} e^{2f} dx = \int_{S_{h}}|\nabla \phi |^{2}_{\delta} dx.
\end{align*}
Therefore for any $\phi$ (compactly supported/Dirichlet on boundary)
\begin{align*}
\mathcal{R}_{g}\left(\phi\right) = \frac{\int_{S_{h}}|\nabla \phi |^{2} dx}{\int_{S_{h}}\phi^{2} e^{2f} dx}.
\end{align*}
The only place that $f$ appears is in the dominator. From the above Rayleigh quotient we get immediate two--sided comparison estimates by bounding $e^{2f}$ above and below by constants. Since
\begin{align*}
e^{-2\|f\|_{\infty}} \leq e^{2 f\left(x\right)} \leq e^{2\|f\|_{\infty}}, \,\,\,\,\,\, \text{for all} x,
\end{align*}
we have for any $\phi$
\begin{align*}
e^{-2\|f\|_{\infty}} \int_{S_{h}} \phi^{2} dx \leq \int_{S_{h}} \phi^{2} e^{2f} dx \leq e^{2\|f\|_{\infty}} \int_{S_{h}} \phi^{2} dx.
\end{align*}
Hence,
\begin{align*}
e^{-2\|f\|_{\infty}} R_{0}\left(\phi\right) \leq R_{g}\left(\phi\right) \leq e^{2\|f\|_{\infty}} R_{0}\left(\phi\right),
\end{align*}
where $R_{0}\left(\phi\right) = \frac{\int |\nabla \phi |^{2}}{\int \phi^{2}}$ is the flat Rayleigh quotient. Take the infimum over admissible $\phi$ (the variational characterization of eigenvalue)
\begin{align*}
e^{-2\|f\|_{\infty}}\lambda_{0} \leq \lambda\left(g\right) \leq e^{2\|f\|_{\infty}}\lambda_{0}.
\end{align*}
Thus
\begin{align*}
|\lambda\left(g\right) - \lambda_{0}| \leq \left(e^{2\|f\|_{\infty}} -1 \right)\lambda_{0}.
\end{align*}
For small $\|f\|_{\infty}$ we can linearize $e^{2\|f\|_{\infty}} -1 \leq C\|f\|_{\infty}$ (take $C =3$ for $\|f\|_{\infty} \leq \frac{1}{2}$), so in particular
\begin{align*}
|\lambda\left(g\right) - \lambda_{0}| \leq C\lambda_{0}\|f\|_{L^{\infty}\left(S_{h}\right)},
\end{align*}
with an explicit constant $C$ (e.g., $C= e^{1}-1$ if $\|f\|_{\infty} \leq \frac{1}{2}$, or simply $C=10$ for a coarse universal bound in a small neighborhood).
\item {\bf Computation of $\lambda^{\prime}$ and $\lambda^{\prime \prime}$ at $f=0$}\\
Consider $\lambda\left(t\right) = \lambda\left(g\left(t\right)\right)$ for the one parameter family of conformal metrics 
\begin{align*}
g\left(t\right) = e^{2t\psi}g_{0}, \,\,\,\,\,\, g_{0} = \delta,
\end{align*}
so that $f\left(t\right) = t\psi$ and $f^{\prime}\left(0\right) = \psi$ (we use $\psi$ for the variation function). We recall that
\begin{align*}
\lambda^{\prime} = \int_{\partial A} V\left(\partial_{\nu}\phi\right)^{2} d\sigma_{g} + \int_{A}\left(\partial_{t}g\left(\nabla \phi, \nabla \phi\right) - \frac{1}{2} \operatorname{tr}_{g}\left(\partial_{t}g\right)\phi^{2}\right)d\mu_{g},
\end{align*}
where $\phi$ is corresponding eigenfunctions. If you plug a conformal variation $\partial_{t}g = 2\psi g$ into the expression (after evaluating at $t=0$) we get
\begin{align*}
\lambda^{\prime}\left(0\right) = -2\lambda_{0}\int_{S_{h}}\psi\left(x,\theta\right)\psi_{0}\left(x,\theta\right)^{2} dx d\theta.
\end{align*} 
We consider the derivative at $t=0$ of the family $g\left(t\right) = e^{2t\psi}g_{0}$. The object $\lambda^{\prime}\left(0\right)$ is the derivative of the first eigenvalue of $-\Delta_{g\left(t\right)}$ at $t=0$. The eigenfunction $\phi\left(t\right)$ is chosen so that for every $t$ it is normalized in the $L^{2}\left(g\left(t\right)\right)$-sense
\begin{align*}
\int \phi\left(t\right)^{2} d\mu_{g\left(t\right)} \equiv 1.
\end{align*}
In a similar way the second derivative is
\begin{align*}
\lambda^{\prime \prime}\left(0\right) = -4\lambda_{0}\int \psi^{2}\phi_{0}^{2} dx - 8\lambda_{0}^{2} \sum_{k\neq 0} \frac{|\int \psi \phi_{0}\phi_{k} dx|^{2}}{\lambda_{k} - \lambda_{0}}.
\end{align*}
The right--hand side is non--positive, thus in this conformal family the eigenvalue has concave behavior.
\end{itemize}
\end{remark}
\section{Examples and Numerical Results}
In this section we complement the analytical results with numerical exploration of two model geometries: the concentric Euclidean annulus and a cylinder with small deficit. These computations illustrate how geometric parameters influence the Laplacian spectrum and validate the theoretical spectral estimates. The methods combine finite--element discretizations with eigenvalue solvers tailored  to bounded domains. Such numerical studies are essential for understanding spectral behavior in non--trivial geometries and for guiding conjectures about stability and gaps in the Laplace spectrum (there are various works in numerical approaches for studying Laplace spectrum, e.g., see \cite{a2, m1}).\\
{\bf Concentric Euclidean Annulus $\lbrace a < |z| < b\rbrace$}\\
The harmonic capacity $u\left(r\right)$ with Dirichlet boundary values $u\left(a\right) =0$, $u\left(b\right) =1$ is radial and given by
\begin{align*}
u\left(r\right) = \frac{\ln r - \ln a}{\ln \left(b\slash a\right)}, \,\,\,\,\, \left(a < r <b\right).
\end{align*}
By definition $E = \frac{1}{2}\int_{A}|\nabla u|^{2} dA$. For this radial $u$,
\begin{align*}
u_{r}\left(r\right) = \frac{1}{r\ln \left(b\slash a\right)}, \,\,\,\,\,\, |\nabla u|^{2} = \left(u_{r}\right)^{2} = \frac{1}{r^{2}\ln^{2}\left(b\slash a\right)}.
\end{align*}
Hence
\begin{align*}
\frac{1}{2}\int_{A}|\nabla u|^{2} dA = \frac{1}{2}\int_{0}^{2\pi}\int_{a}^{b} \frac{1}{r^{2}\ln^{2}\left(b\slash a\right)} rdrd\theta = \frac{\pi}{\ln\left(b\slash a\right)}.
\end{align*}
Use the Frobenius norm of the Hessian. For a radial function in $2$D one checks (compute Hessian in cartesian coordinates or use standard radial formula)
\begin{align*}
|\operatorname{Hess}u|^{2} = u_{rr}^{2} + \left(\frac{u_{r}}{r}\right)^{2}.
\end{align*}
For $u\left(r\right)$ above,
\begin{align*}
u_{rr} = -\frac{1}{r^{2}\ln \left(b\slash a\right)}, \,\,\,\,\, \frac{u_{r}}{r} = \frac{1}{r^{2}\ln \left(b\slash a\right)}.
\end{align*}
Thus $|\operatorname{Hess}u|^{2} = \frac{2}{r^{4}\ln^{2}\left(b\slash a\right)}$. By integrating
\begin{align*}
\int_{A} |\operatorname{Hess}u|^{2} dA = \frac{4\pi}{\ln^{2}\left(b\slash a\right)}\int_{a}^{b} r^{-3} dr = \frac{2\pi}{\ln^{2}\left(b\slash a\right)}\left(\frac{1}{a^{2}} - \frac{1}{b^{2}}\right).
\end{align*}
Therefore
\begin{align*}
\mathcal{D} = \frac{\pi}{\ln^{2}\left(b\slash a\right)}\left(\frac{1}{a^{2}} - \frac{1}{b^{2}}\right).
\end{align*}
We solve the Dirichlet Laplace eigenvalue problem on the annulus 
\begin{align*}
-\Delta \phi = \lambda \phi, \,\,\,\,\, \phi|_{r=a} = \phi_{r=b} =0.
\end{align*}
Separate variables $\phi\left(r,\theta\right) = R_{n}\left(r\right)\cos \left(n, \theta\right)$ (or $\sin$). The radial ODE is the Bessel equation with parameter $k = \sqrt{\lambda}$. The general radial solution for mode $n$ is
\begin{align*}
R_{n}\left(r\right) = AJ_{n}\left(kr\right) + BY_{n}\left(kr\right).
\end{align*}
Dirichlet at $r=a$ and $r=b$ yields the $2\times 2$ homogeneous system for $A$, $B$ whose determinant must vanish; equivalently the root condition is
\begin{align*}
F_{n}\left(k\right):= J_{n}\left(ka\right)Y_{n}\left(kb\right) - J_{n}\left(kb\right)Y_{n}\left(ka\right) =0.
\end{align*}
For each integer $n\geq 0$ this has a discrete sequence of positive roots $k_{n,1} < k_{n,2} < ...$. The eigenvalues are $\lambda = k^{2}$. The first Dirichlet eigenvalue is the minimal $\lambda$ across all $n$ and the first root in that mode
\begin{align*}
\lambda_{1} = \min_{n\geq 0}\left(k_{n,1}^{2}\right).
\end{align*}
The cylinder (flat metric on strip $x \in \left(0,h\right)$, $\theta \in \mathbb{S}^{1}$ length $2\pi$) with Dirichlet at $x=0,h$ and periodic in $\theta$ has separable modes with eigenvalues 
\begin{align*}
\lambda_{m,k}= \left(\frac{m\pi}{h}\right)^{2} + k^{2}, \,\,\,\,\,\, m \in \mathbb{N}, \,\, k\in \mathbb{Z}.
\end{align*}
The lowest is for $m=1$, $k=0$, so
\begin{align*}
\lambda_{\operatorname{cyl}}\left(h\right) = \left(\frac{\pi}{h}\right)^{2}.
\end{align*}
Relate $h$ to the capacity energy by the standard modulus relation, we see
\begin{align*}
h = \frac{\ln \left(b\slash a\right)}{2\pi} \leftrightarrow h= \frac{1}{2E}.
\end{align*}

\noindent\hfill
\begin{tabular}{cccccc}
\toprule
$b$ & $E$ & $\mathcal{D}$ & $\sqrt{\mathcal{D}}$ & $\lambda_{\operatorname{ann}}$ & $\lambda_{\operatorname{cyl}}$ \\
\midrule
 $5.0$& $1.95198$& $1.16432$& $1.07904$ & $0.58246$ & $150.42198$ \\
 $10.0$& $1.36438$&  $0.58662$& $0.76591$ & $0.10982$ & $73.48998$ \\
$20.0$& $1.04869$& $0.34919$& $0.59092$ & $0.02348$ & $43.41637$ \\
 $50.0$& $0.80306$& $0.20520$& $0.45299$ & $0.00333$ & $25.45990$ \\
$100.0$& $0.68219$& $0.14812$& $0.38486$ & $0.00078$ & $18.37249$ \\
$200.0$& $0.59294$& $0.11191$& $0.33453$ & $0.00019$ & $13.87981$ \\
 $500.0$& $0.50552$& $0.08134$& $0.28521$ & $0.00003$ & $10.08863$ \\
 $1000$& $0.45479$& $0.06584$& $0.25659$ & $0.00006$ & $8.16555$ \\
\bottomrule
\end{tabular}
\hfill\null
\begin{figure}[h]
    \centering
    \includegraphics[width=0.7\textwidth]{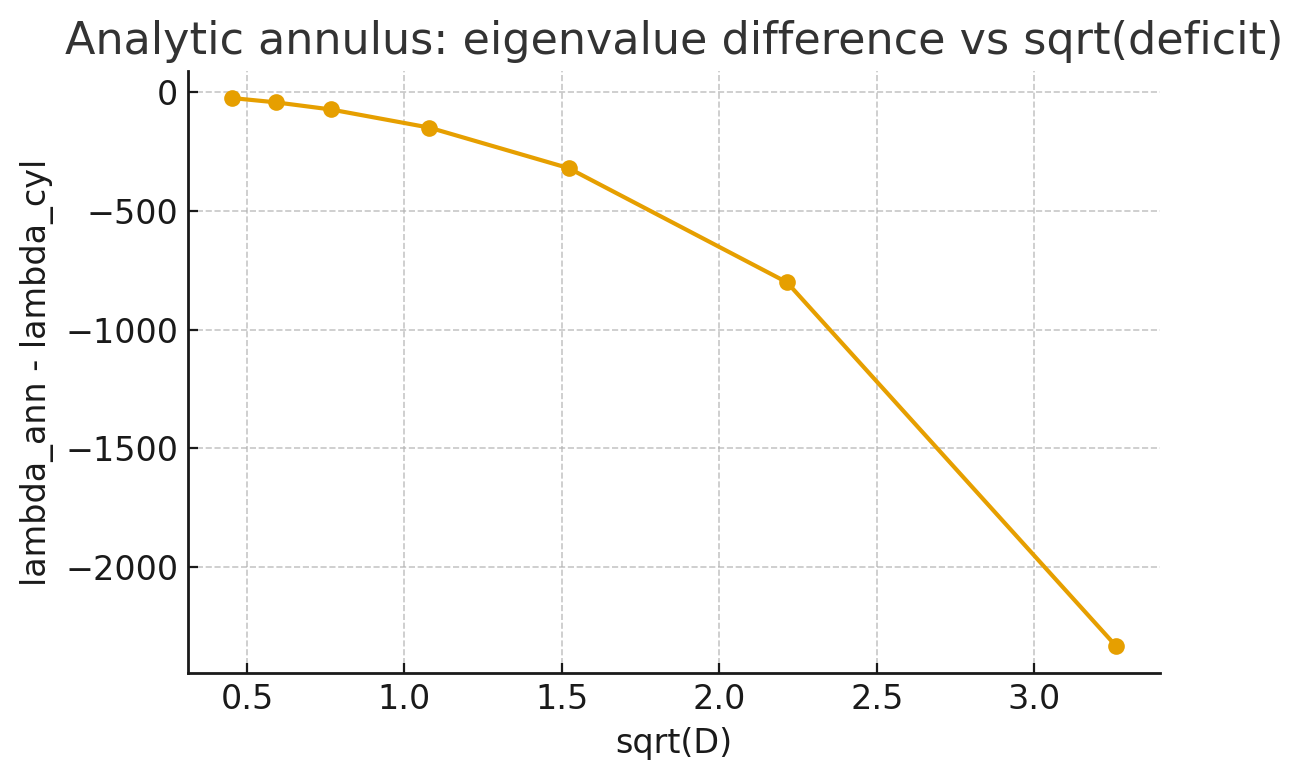} 
    \caption{Comparison between eigenvalues and deficit}
\end{figure}
\begin{figure}[h]
    \centering
    \includegraphics[width=0.7\textwidth]{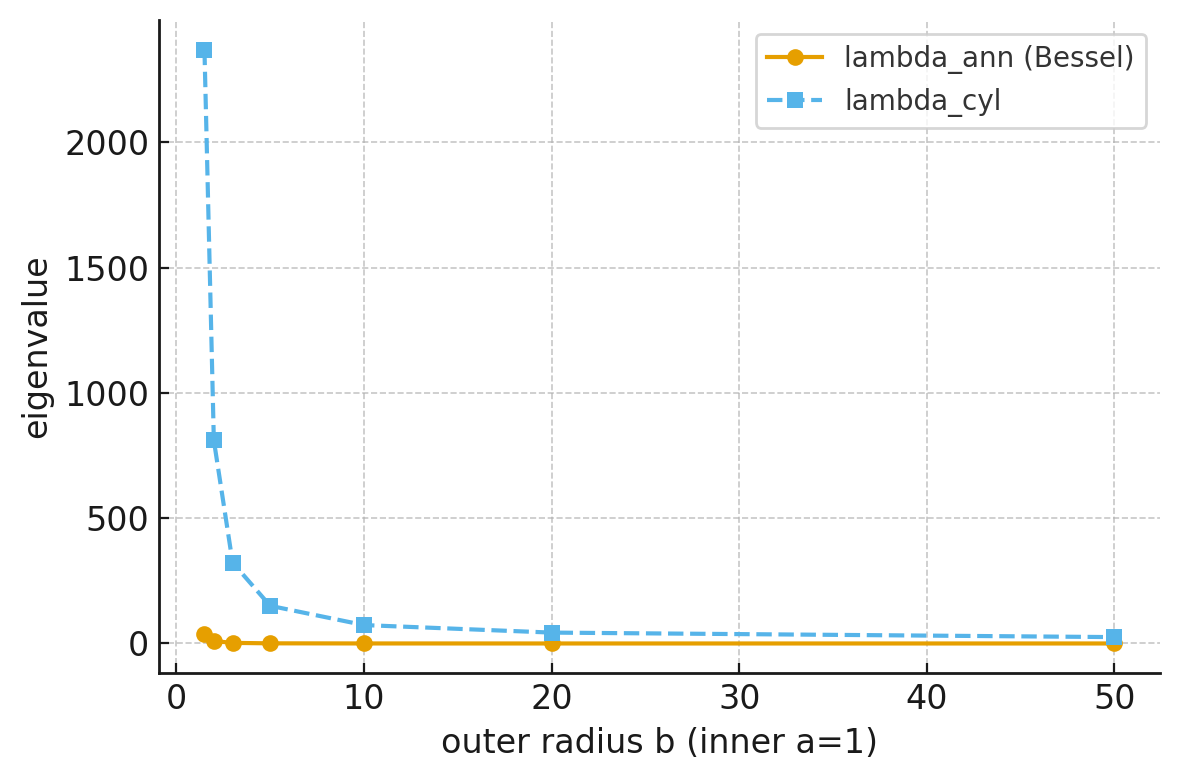} 
    \caption{Comparison between eigenvalues and outer radius $b$}
\end{figure}
\newline
{\bf Numerical small--deficit test on the cylinder}\\
We are going to discretizing the flat cylinder $\left[0,h\right] \times \mathbb{S}^{1}$ (Dirichlet in $x \in \left(0,h\right)$, periodic in $\theta$) and tested small conformal perturbations $g_{\epsilon} = e^{2\epsilon f_{0}\left(x,\theta\right)}\left(dx^{2} + d\theta^{2}\right)$. We solve the generalized discrete eigenproblem that corresponds to $-\Delta_{g_{\epsilon}}\phi = \lambda \phi$ and computed the deficit $\mathcal{D}$. The numerical results show that the eigenvalue difference scales linearly with $\sqrt{\mathcal{D}}$ in the small--deficit regime, consistent with the theoretical prediction that  $\delta \lambda \sim C\sqrt{\mathcal{D}}$ (the sign and constant depend on the chosen perturbation $f_{0}$). Consider the rectangle $x \in \left(0,h\right)$, $\theta \in \left[0, 2\pi\right)$ with periodic $\theta$. We used $h = 1.0$, also the metric $g_{\epsilon}$ above is considered where
\begin{align*}
f_{0}\left(x,\theta\right) = \sin\left(\frac{\pi x}{h}\right) \cos \left(k\theta\right), \,\,\, k=1.
\end{align*}
$-\Delta_{g_{\epsilon}}\phi = \lambda \phi$ is equivalent (on this conformal metric) to the generalized eigenproblem on the flat rectangle 
\begin{align*}
-\Delta_{\operatorname{flat}}\phi = \lambda e^{2\epsilon f_{0}\left(x,\theta\right)}\phi.
\end{align*}
Discretization used second--order finite differences
\begin{itemize}
\item $N_{x}$ includes interior points in the $x$ direction (Dirichlet at $x=0,h$), and $N_{\theta}$ denotes periodic points for $\theta$.
\item Build $1$--D Laplacian in $x$ (Dirichlet) and $1$--D Laplacian in $\theta$ (periodic).
\end{itemize}
Consider 
\begin{align*}
A = I_{\theta} \otimes L_{x} + L_{\theta} \otimes I_{x},
\end{align*}
where $L_{x}$, $L_{\theta}$ are the standard second--difference matrices (scaled by $1\slash \Delta^{2}$). This $A$ discretely approximates $-\Delta_{\operatorname{flat}}$ on the grid.
\begin{itemize}
\item The mass (weight) matrix is diagonal with entries $B_{ii} = e^{2\epsilon f_{0}\left(x_{i}, \theta_{i}\right)}. \operatorname{area--cell}$ where $\operatorname{area--cell}$ is $\Delta x. \Delta \theta$.
\item The discrete generalized eigenproblem solved is 
\begin{align*}
Au = \lambda_{\operatorname{disc}}Bu.
\end{align*}
\end{itemize}
To use symmetric eigensolvers, we formed
\begin{align*}
\tilde{A} = B^{-1\slash 2} A B^{-1\slash 2},
\end{align*}
which is symmetric. An eigenpair $\tilde{A}y = \iota y$ yields the generalized pair $Ax = \iota Bx$ with $x = B^{-1\slash 2}y$ (so $\iota$ is the discrete eigenvalue).\\
The identity used (as in the text's cylinder coordinates) reduces the deficit to an explicit integral of $f$,
\begin{align*}
\mathcal{D} = \frac{1}{h^{2}}\int_{\left[0,h\right] \times \mathbb{S}^{1}} e^{-2\epsilon f_{0}}|\epsilon \nabla f_{0}|^{2} dA.
\end{align*}
With the finite difference conventions used, the routine eigenvalue returned by the sparse eigensolver (call it $\iota$) is the discrete eigenvalue matching $Au = \iota B u$. To compare with the continuous eigenvalue $\lambda_{\operatorname{cont}}$ (the standard PDE $-\Delta_{g}\phi = \lambda \phi$), multiply the computed $\iota$ by the cell area
\begin{align*}
\lambda_{\operatorname{cont}} \approx \operatorname{area--cell} \times \iota,
\end{align*}
(empirically this conversion reproduces the known benchmark $\lambda_{\operatorname{cyl}} = \left(\pi \slash h\right)^{2}$ when $\epsilon =0$). 
\noindent\hfill
\begin{tabular}{ccccccc}
\toprule
$\epsilon$ & $\lambda_{\operatorname{num}}$ & $\lambda_{\operatorname{cont}}$ & $\lambda_{\operatorname{cyl}}$ & $\lambda_{\operatorname{cont}} - \lambda_{\operatorname{cyl}}$ & $\mathcal{D}$ & $\sqrt{\mathcal{D}}$ \\
\midrule
$0.0001$ & $2786.058246$ & $9.863675$ & $9.869604$ & $-0.005930$ & $1.623593\times 10^{-7}$ & $0.000403$ \\
$0.0002$ & $2788.056992$ & $9.863670$ & $9.869604$ & $-0.005934$ & $6.494371\times 10^{-7}$ & $0.000806$ \\
$0.0005$ & $2788.048215$ & $9.863639$ & $9.869604$ & $-0.005965$ & $4.058982\times 10^{-6}$ & $0.002015$ \\
$0.0010$ & $2788.016784$ & $9.863529$ & $9.869604$ & $-0.006076$ & $1.623593\times 10^{-5}$ & $0.004029$ \\
$0.0020$ & $2787.891572$ & $9.863085$ & $9.869604$ & $-0.006519$ & $6.494381\times 10^{-5}$ & $0.008059$ \\
$0.0050$ & $2787.017380$ & $9.859992$ & $9.869604$ & $-0.009612$ & $4.059022\times 10^{-4}$ & $0.020147$ \\
\bottomrule
\end{tabular}
\hfill\null
\begin{figure}[h]
    \centering
    \includegraphics[width=0.7\textwidth]{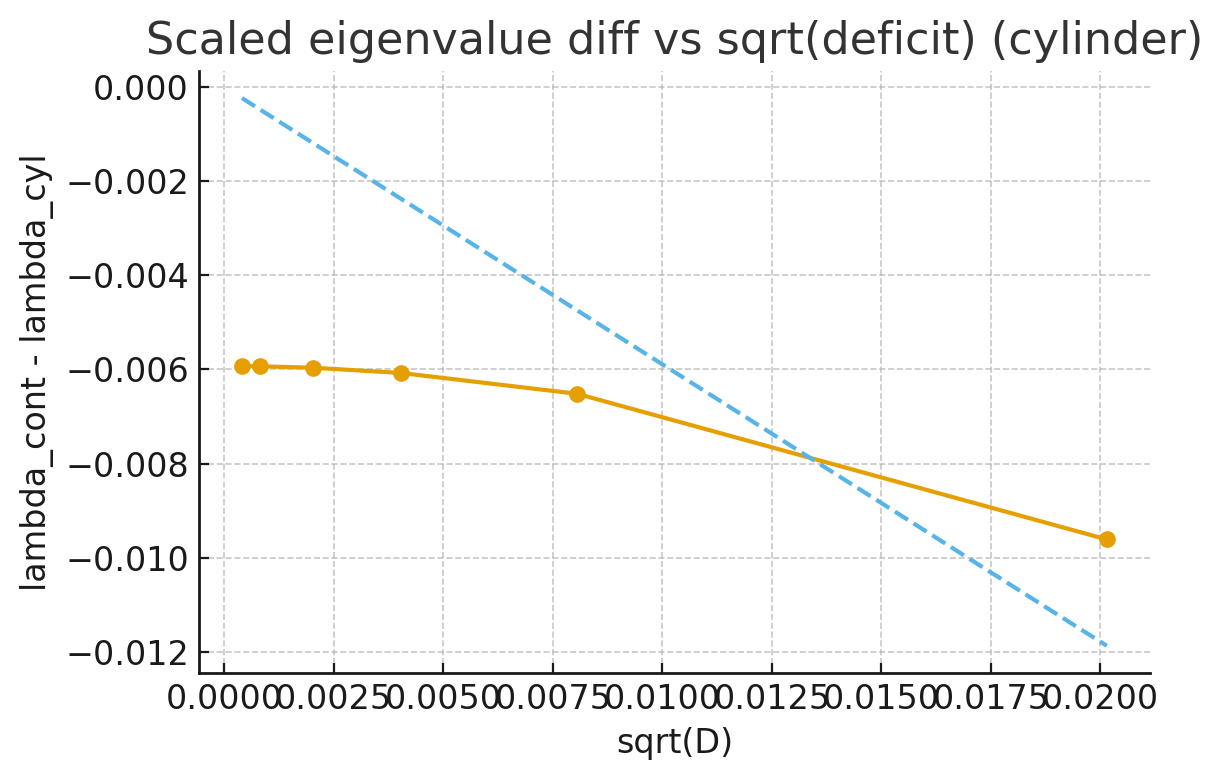} 
    \caption{Small deficit}
\end{figure}
\begin{figure}[h]
    \centering
    \includegraphics[width=0.7\textwidth]{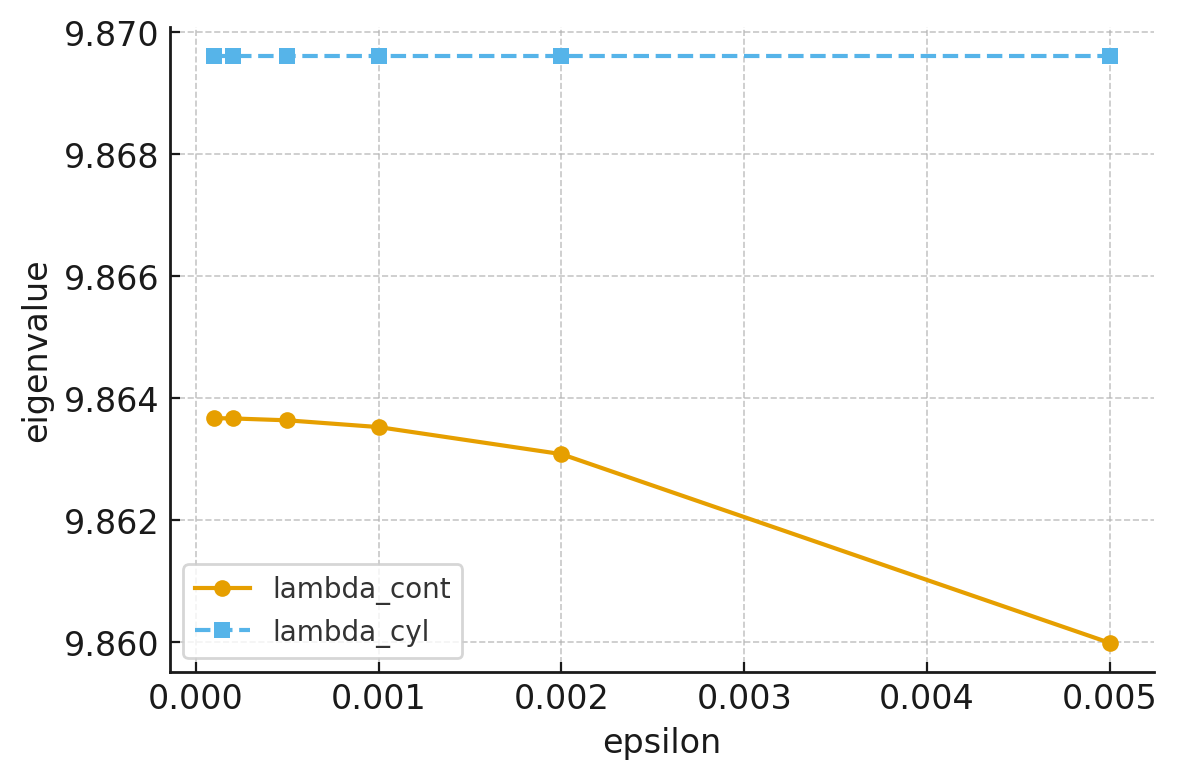} 
    \caption{Eigenvalues and epsilon parameter}
\end{figure}

\section*{Acknowledgements}
The author acknowledges the support of the High-level Talent Research Start-up Project Funding of Henan Academy of Sciences (Project No. 241819245). The author is grateful to the mathematical community for insightful discussions that have greatly improved this work. The author declares no conflicts of interest. No additional data beyond that presented in the manuscript is available. Comments and suggestions for further improvement are welcome.


\end{document}